\documentclass[12pt]{amsart}
\usepackage[obeyspaces,hyphens,spaces]{url}
\usepackage{amsmath, amssymb, amsthm}
\usepackage{comment, empheq,multicol}
\usepackage{mathdots, breqn, xcolor}
\usepackage[colorlinks]{hyperref}
\usepackage{mathtools}
\newcommand{\parallelslant}{\mathbin{/\mkern-4mu/}}
\usepackage[capitalize,nameinlink,noabbrev]{cleveref}
\usepackage{bbm}
\usepackage{dsfont}
\usepackage{todonotes}
\numberwithin{equation}{section}
\newtheorem{theorem}{Theorem}
\newtheorem{lemma}{Lemma}[section]
\newtheorem{prop}[lemma]{Proposition}

\newtheorem{cor}[lemma]{Corollary}

\theoremstyle{definition}

\newtheorem{remark}[lemma]{Remark}

\crefname{subsection}{Subsection}{Subsections}

\newcommand{\R}{\mathbb{R}}

\newcommand{\C}{\mathbb{C}}

\newcommand{\Z}{\mathbb{Z}}
\renewcommand{\O}{\mathrm{O}}
\newcommand{\I}{\mathrm{I}}
\newcommand{\N}{\mathbb{N}}

\newcommand{\X}{\mathbb{X}}

\newcommand{\g}{\mathfrak{g}}

\renewcommand{\o}{\mathrm{o}}

\newcommand{\GL}{\mathrm{GL}}
\newcommand{\PGL}{\mathrm{PGL}}
\newcommand{\SL}{\mathrm{SL}}

\newcommand{\gen}{\mathrm{gen}}

\newcommand{\diag}{\mathrm{diag}}
\newcommand{\aut}{\mathrm{aut}}
\newcommand{\tr}{\mathrm{tr}}

\newcommand{\supp}{\mathrm{supp}}
\newcommand{\dist}{\mathrm{dist}}

\newcommand{\Op}{\mathrm{Op}}

\newcommand{\calO}{\mathcal{O}}
\newcommand{\calW}{\mathcal{W}}

\newcommand{\cusp}{\mathrm{cusp}}

\newcommand{\Hom}{\mathrm{Hom}}

\newcommand{\vol}{\mathrm{vol}}
\renewcommand{\d}{\,\mathrm{d}}

\usepackage{svg}
\usepackage[color,notref,notcite,final]{showkeys}
\renewcommand*{\showkeyslabelformat}[1]{%
\fbox{\vbox{\hsize=1.1cm\normalfont\small\url{#1}\par}}}

\title[Sarnak's density hypothesis in the spectral aspect]{On the spectral aspect density hypothesis and application}

\author{Edgar Assing}
\address{Mathematical Institute of the University of Bonn, Endenicher Allee 60, D–53115 Bonn, Germany}
\email{assing@math.uni-bonn.de}

\author{Subhajit Jana}
\address{Queen Mary University of London, Mile End Road, London E1 4NS, UK.}
\email{s.jana@qmul.ac.uk}

\thanks{The first author is supported by the Germany Excellence Strategy grant EXC-2047/1-390685813 and also partially funded by the Deutsche Forschungsgemeinschaft (DFG, German Research Foundation) – Project-ID 491392403 – TRR 358.}

\begin{document}

\begin{abstract}
    We prove that the density of non-tempered (at any $p$-adic place) cuspidal representations for $\mathrm{GL}_n(\mathbb{Z})$, varying over a family of representations ordered by their infinitesimal characters, is small -- confirming Sarnak's density hypothesis in this set-up. Among other ingredients, the proof uses tools from microlocal analysis for Lie group representations as developed by Nelson and Venkatesh. As an application, we prove that the Diophantine exponent of the $\mathrm{SL}_n(\mathbb{Z}[1/p])$-action on $\mathrm{SL}_n(\mathbb{R})/\mathrm{SO}_n(\mathbb{R})$ is \emph{optimal} -- resolving a conjecture of Ghosh, Gorodnik, and Nevo.
\end{abstract}

\maketitle

\section{Introduction}

The Generalized Ramanujan Conjecture (GRC) has many important applications in analytic number theory. However, since a complete proof of the GRC seems to be illusive with current technology, many arithmetic results only remain conditional. A good approximation to the GRC, which is more tractable using tools from the spectral theory of automorphic forms, is Sarnak's density hypothesis. In fact, this hypothesis turns out to be sufficient for many practical purposes. More precisely, the density hypothesis is a quantitative statement about the rareness of violations of the GRC in families of automorphic representations. The underlying family is usually dictated by the application at hand.

For example, the resolutions of the \textit{uniform counting} conjecture and the \textit{optimal lifting} conjecture require a density hypothesis for a certain level-aspect family determined by the principal congruence subgroup. These density results have recently been proved in \cite{assing2024principal, assing2024density} building on earlier work \cite{blomer2019density} for the Hecke congruence subgroup. On the other hand, the question regarding the \textit{optimal Diophantine exponent} (see Section~\ref{sec:optimal} for details) requires a certain density hypothesis for a spectral-aspect family. This was pointed out in \cite{jana2024optimal}. The main goal of this paper is to prove a density hypothesis which, in consequence, yields the required one for the optimal exponent conjecture.

The natural starting point for proving any density result is a suitable spectral summation formula. A key step in this approach is to derive strong enough estimates for (twisted relative) orbital integrals. This is a local problem, which is particularly difficult in the archimedean aspect. Approaching these questions using the Selberg Trace formula has seen some success for $\GL_2$; most notably see \cite{fraczyk2024density}. However, for higher rank these techniques often fall short of the density hypothesis; see \emph{e.g.}, \cite{matz2021sato-tate}. On the other hand, much of the recent progress relies on the use of the Kuznetsov formula instead, as pioneered by Iwaniec in \cite{iwaniec1990small}. Following this approach the density hypothesis for a spectral family in its full strength (and beyond) was first established for $\GL_3$ in \cite{blomer2014sato}; see also \cite[Proposition~4.15]{jana2024optimal}. For $\GL_n$ with $n\ge 4$ the only available result in the archimedean aspect is obtained in \cite{jana2021application}, where the density hypothesis is established for the archimedean conductor family.

In this paper, we study the spectral family $\Omega_{\textrm{cusp}}(X)$ of all (cuspidal) automorphic representations with spectral parameter in a (growing) ball of radius $X$; see \eqref{eq:def_family} below for a precise definition, and the corresponding density theorem is Theorem~\ref{th:main_density} below. Working with this family turns out to be much more difficult than working with conductor families as in \cite{blomer2019density, jana2021application}; see end of \S\ref{sec:set-up-proof} for a discussion on this. In fact, we encounter difficulties along the lines of \cite{assing2024principal, assing2024density} from where we borrow some ideas. We refer to Section~\ref{sec:method} below for a explanation of our method and its relation to the non-archimedean situation.

\subsection{The main result}

Let $n\ge 2$ and $\X:=\X_n:=\PGL_n(\Z)\backslash \PGL_n(\R)$. For $X\ge 1$, let $\Omega_{\textrm{cusp}}(X)$ denote the isomorphism class of cuspidal automorphic representations $\pi$ that appear in the spectral decomposition of the space $L^2(\X)$ and whose (archimedean) spectral parameters $\mu_{\infty}(\pi)$ lie in a ball of radius $X$. It is important to note that the family $\Omega_{\textrm{cusp}}(X)$ contains non-spherical constituents. Recall from \cite{muller2007weyl} that the Weyl law for the spherical sub-family reads
\begin{equation}
    \# \{ \pi \in \Omega_{\textrm{cusp}}(X) \colon \pi \text{ is spherical}\} \asymp X^{\frac{(n+2)(n-1)}{2}}.\nonumber
\end{equation}
The full family $\Omega_{\textrm{cusp}}(X)$ is strictly larger, but its size is of the same order of magnitude. Indeed, one can verify the essentially sharp upper bound
\begin{equation}
    \# \Omega_{\textrm{cusp}}(X) \ll_\epsilon X^{\frac{(n+2)(n-1)}{2}+\epsilon}. \label{eq:weyl}
\end{equation}

Given a prime $p$, for each $\pi$ we associate the $p$-adic Satake parameters $$\mu_p(\pi):=\{\mu_p(\pi)_1,\ldots,\mu_p(\pi)_n\}.$$ This is a (multi)-set of complex numbers uniquely determined by the local $L$-factors $L_p(s,\pi)$ of $\pi$ at $p$. Since $\pi$ is assumed to be cuspidal the GRC predicts that $\pi$ is everywhere tempered. In particular, we expect that the elements of $\mu_p(\pi)$ are purely imaginary. Our goal is to quantify the statement that 
$$\textit{cuspidal representations that are highly non-tempered at $p$ are rare}.$$
To do so, we first introduce the a measurement of non-temperedness, namely,
\begin{equation}
    \sigma_p(\pi) := \max_{i=1,\ldots, n} \vert \Re(\mu_p(\pi)_i)\vert.\nonumber
\end{equation}
With this notation at hand, we are ready to state our main theorem, which is an archimedean analogue of \cite[Theorem~1.1]{assing2024density}, anticipated in \cite[Section~1.5 Point (3)]{assing2024density}.

\begin{theorem}\label{th:main_density}
Fix $n\geq 4$. Let $p$ be any prime and $\sigma\ge 0$. We have
\begin{equation}
    \# \{\pi\in \Omega_{\mathrm{cusp}}(X)\colon  \sigma_p(\pi)\geq \sigma\} \ll_{p,\epsilon} X^{\frac{1}{2}(n+2)(n-1) \left(1-\frac{4\sigma}{n-1}\right)+\epsilon}. \nonumber
    \end{equation}
for all sufficiently large $X$.
\end{theorem}

This theorem gives a power-saving counting of cusp forms on $\Omega_\cusp(X)$ that are at least ``$\sigma$-non-tempered'' at $p$. Following the Sarnak's philosophy \cite{sarnak1991bounds} a reasonable expectation is obtained by interpolating between the Weyl law (see \eqref{eq:weyl}) and the multiplicity one for the trivial representation. This way one obtains a density hypothesis of the form
\begin{equation}
    \# \{\pi\in \Omega_{\textrm{cusp}}(X)\colon  \sigma_p(\pi)\geq \sigma\} \ll_{p,\epsilon} X^{\frac{1}{2}(n+2)(n-1) \left(1-\frac{2\sigma}{n-1}\right)+\epsilon}. \nonumber
\end{equation}
Comparing this with the estimate established in Theorem~\ref{th:main_density} shows that we have proven a stronger statement. Note that this is only possible when restricting to families of cuspidal automorphic representations.

\begin{remark}
The restriction $n\ge 4$ is purely to keep the paper shorter and focused, as for $n=2,3$, a slightly different treatment of the \emph{orbital integrals} will be required. In any case, for $n=2$ and $n=3$ the ``spherical variants'' of this density theorem are known. For $n=2$, we refer to \cite{humphries2018density} for the corresponding theorem. For $n=3$, the relevant theorem is stated in \cite[Proposition~4.15]{jana2024optimal}.     
\end{remark}

\subsection{Application}\label{sec:optimal}

Here we briefly discuss the set-up of Ghosh--Gorodnik--Nevo, as in \cite{ghosh2015diophantine} and \cite{ghosh2018best}. We also refer to a recent work \cite{fraczyk2024optimal} and references therein for general discussion on the recent progress in this direction.

Recall the definition of \emph{Diophantine exponent} $\kappa$ from \cite[Definition 1.1]{jana2024optimal}. Loosely speaking, $\kappa$ is the infimum of all $\zeta$ such that for almost every $x,x_0\in \SL_n(\R)/\mathrm{SO}_n(\R)$ and for all sufficiently small $\varepsilon>0$ there exists a $\gamma\in \SL_n(\Z[1/p])$ with \emph{denominator} bounded by $\varepsilon^{-\frac{n+2}{2n}\zeta}$ satisfying $\dist(x,\gamma x_0)<\varepsilon$. By a covering argument one can show that $\kappa\ge 1$. In \cite{ghosh2018best}, it is implicitly stated that one should expect $\kappa\le 1$, and consequently, $\kappa=1$ \emph{optimal}. Informally, optimality of $\kappa$ is a manifestation of the Dirichlet-type result in the set-up of the (intrinsic) Diophantine approximation for the group $\SL_n$.

In \cite{ghosh2018best}, it is shown that $\kappa\le n-1$.
In the work of the second author and Kamber \cite[Theorem 2]{jana2024optimal}, they showed that $\kappa=1$ for $n=2,3$ and 
$$\kappa \le 1+ \frac{2\theta_n}{n-1-2\theta_n},$$
for $n\ge 4$, where $\theta_n$ is any bound towards temperedness for for the cuspidal spectrum of $\GL_n$. We remark that even if there is only one non-tempered cusp form the method of \cite{jana2024optimal} will be unable to prove optimality of $\kappa$.

On the other hand, it is also proved in \cite[Theorem~3]{jana2024optimal} that if certain version of Sarnak's density hypothesis is true, namely \cite[Conjecture~2]{jana2024optimal}, then $\kappa$ is optimal. Our main result Theorem \ref{th:main_density} directly implies \cite[Conjecture~2]{jana2024optimal}.

\begin{theorem}\label{th:optimal-exp}
For all $n\geq 4$ and prime $p$, the Diophantine exponent in the context of approximation of points on $\SL_n(\R)/\mathrm{SO}_n(\R)$ by $\SL_n(\Z[1/p])$ is optimal, \emph{i.e.}, $\kappa=1$.
\end{theorem}

\begin{remark}
Theorem \ref{th:optimal-exp} can be thought as the archimedean counterpart of Sarnak's optimal lifting conjecture, which is recently resolved by \cite{assing2024principal} and \cite{jana2024local-l2} (also \cite{assing2024density}).
\end{remark}

\subsection{Set-up of the proof}\label{sec:set-up-proof}

Let us briefly describe the main ideas going in our proof of Theorem \ref{th:main_density}. To do so it will require us to use some (mostly standard) notation. We refer to Section~\ref{sec:not} for precise definitions.

We first note that instead of proving the counting result as formulated in Theorem~\ref{th:main_density} we will prove that for $l\in\Z_{\ge 0}$ we have
\begin{equation}\label{eq:Hecke_Version}
    \sum_{\pi \in \Omega_{\textrm{cusp}}(X)} \left\vert \lambda_{\pi}\left(p^l\right)\right\vert^2 \ll_{p,\epsilon} X^{\frac{1}{2}(n+2)(n-1) + \epsilon}, \quad\text{if } p^l\leq X^{n+2}.
\end{equation}
Here $\lambda_{\pi}(\cdot)$ are certain Hecke eigenvalues attached to $\pi$, in particular, they appear as the coefficients of the Dirichlet series for $L(\cdot,\pi)$. A standard argument, using ``Rankin's trick'', shows that this is sufficient; see for example \cite[Proposition~4.12]{jana2024optimal} or Section~\ref{sec:proof} for details.

As in previous works, we will use the Kuznetsov formula as our main global tool for the proof of \eqref{eq:Hecke_Version}. Given a test function $F$ on $\PGL_n(\R)$ it roughly states
\begin{equation}
    \sum_{\widehat{\X}_\gen} \vert \lambda_{\pi}(m)\vert^2 J_{\pi}(F) \approx \sum_{w\in W}\sum_{c\in \Z_{\neq 0}^{n-1}}S_{w}(m,m;c)\mathcal{O}_{w}(F,m,c),\nonumber
\end{equation}
where $J_\pi(F)$ is the Bessel distribution attached to $\pi$, $S_w(m,m;c)$ is a certain Kloosterman sum attached to the Weyl element $w$ and $\calO$ is a certain orbital integral (\emph{i.e.}, an archimedean version of the Kloosterman sum). See Proposition~\ref{prop:Kuznetsov} below for a precise statement. In order for this (abstract) formula to be useful we have to choose a suitable function $F$ (depending on $X$), such that $J_\pi(F)$ localizes on the set $\Omega_{\textrm{cusp}}(X)$. The problem then becomes to control the orbital integrals on the geometric side of the above formula.

In the level aspect (\emph{aka}, the non-archimedean setting) the families in question are naturally selected by some congruence subgroup. It is the structure of this subgroup, which yields an obvious choice for (the $p$-adic version of) the test function $F$ and ultimately governs the properties of the Kloosterman sums $S_w(m,m;c)$. Such an underlying structure is absent in the spectral setting and there seems to be no obvious way to construct a suitable test function $F$. The classical approach is to restrict to the spherical representation and \textit{guess} a function such that $J_{\pi}(F)$ has the desired properties; see \cite{blomer2013applications} for an approach with similar flavour in $\GL_3$. This function is usually explicitly given in Iwasawa coordinates. As a result one faces significant difficulties when analysing the orbital integrals $\mathcal{O}_w(F,m,c)$. Indeed, when expanding these integrals one encounters difficult expressions involving the Iwasawa projections restricted to the Bruhat cell given by $w$. Directly analysing these expressions for $\GL_n$ with $n$ large is a hurdle too difficult for us to overcome.

\subsection{Our method}\label{sec:method}

In this paper, we turn the tables and construct the test function geometrically following the strategy from \cite{jana2021application}. More precisely, we mimic the non-archimedean situation and choose $F_X$ to be related to a smoothened $L^1$-normalized characteristic function of the set
\begin{equation}
    K(X,\tau) = \{ g\in \GL_n(\R): \|g-\I_n\|_{\infty} \leq \tau/X\},\quad \tau>0\text{ sufficiently small}.\nonumber
\end{equation}
This is an approximate archimedean analogue of the principal congruence subgroup of level $X$. The upshot is that, given our range $p^l\ll X^{n+2}$, the geometric side only features contributions coming from the trivial Weyl element $w=1$ and from the element $w=w_{\ast}$ that also appear in the previous works \cite{blomer2019density,assing2024principal,assing2024density}. Estimating these remaining contributions trivially yields the desired estimate. This analysis is carried out in Section~\ref{sec:orbit} closely following the arguments from \cite[Section~3]{assing2024density}. Note that these arguments rely solely on the support of $F_X$.

The crux of the argument remains to establish a suitable lower bound for $J_{\pi}({F}_X)$ with $\pi\in \Omega_{\textrm{cusp}}(X)$. Note that all $\pi$ in $\Omega_\cusp(X)$ are \emph{generic}. For a cleaner technical discussion, let us define the Bessel distribution attached to $\pi$ with respect to a certain normalized character $\boldsymbol{\psi}_X$ (see \eqref{eq:def_char} below). Based on the $p$-adic theory developed in \cite{assing2024principal, assing2024density} we have the following expectations in the archimedean case. Let $\pi$ be an (infinite dimensional) irreducible representation of $\GL_n(\R)$ with trivial central character. Then the image of $\pi(F_X)$ consists of approximately $K(X,\tau)$-invariant elements.  We further should have:
\begin{enumerate}
    \item\label{item:invariance} A representation $\pi$ features approximately $K(X,\tau)$-invariant vectors (\emph{i.e.}, $\tr(\pi(F_X))$ is not negligible) precisely when $\Vert \mu_{\infty}(\pi)\Vert \ll X$. This is analogous to the bounds on the depth of the representations in the $p$-adic situation.
    \item\label{item:dimension} If $\Vert \mu_{\infty}(\pi)\Vert\ll X$, then the set of approximately $K(X,\tau)$-invariant vectors contains (at most) $\approx X^{\frac{1}{2}n(n-1)}$ mutually orthogonal elements. In other words, we have $\tr(\pi(F_X)) \approx X^{\frac{1}{2}n(n-1)}$ in this case.
\end{enumerate}
These statements can be made precise, we refer to \cite{jana2020analytic} for similar statements. As we will be working with the Kuznetsov (relative) trace formula, as opposed to the Selberg trace formula, we will need to show a ``Whittaker-relative'' version of Item \ref{item:invariance} and Item \ref{item:dimension}. Namely, in Proposition~\ref{prop:spectral_lower_bound} below, we essentially show that the bound 
\begin{equation}\label{eq:lower_bound_Bessel}
    J_{\pi,\boldsymbol{\psi}_X}({F}_X) \gg X^{\frac{1}{2}n(n-1)} \text{ as long as }\Vert \mu_{\infty}(\pi)\Vert\, \leq X 
\end{equation} holds, so that the spectral side of the Kuznetsov formula truthfully reflects the spectral properties of the operator attached to $F_X$ on the cuspidal spectrum. 

Establishing \eqref{eq:lower_bound_Bessel} is the most difficult step and the main novelty in our proof. The corresponding non-archimedean estimate is established in \cite[Section~6]{assing2024principal} and \cite[Section~2]{assing2024density}. While overall the arguments have many parallels, the main hurdles lie at different steps. For example, in the non-archimedean case it is hard to show that automorphic representations appearing in corresponding family feature so called \textit{regular types} at the relevant places. In the archimedean case, the analogous statement (see \eqref{eq:tau} below) is given in \cite[Lemma~7.1 and Lemma~7.8]{nelson2023standard}. The challenging task for us is to construct suitable vectors microlocalized at certain \emph{regular elements} and use their properties to verify \eqref{eq:lower_bound_Bessel}. Here we rely on the corresponding microlocal calculus as developed in \cite{nelson2023standard} building on \cite{nelson2021orbit, nelson2023spectral}.

\subsection{Orbit method heuristics}

We end this introduction by giving heuristic explanation to obtain Item \ref{item:invariance} and Item \ref{item:dimension} using the quantitative orbit method of Nelson--Venkatesh from \cite{nelson2021orbit}. This is, in no way, intended to be a rigorous argument and we apologize in advance for all the approximate identities $\approx$ that will be used. We will also assume that $\pi$ is tempered (at infinity). This allows us to associate an co-adjoint orbit $\mathcal{O}_{\pi}\subseteq \mathfrak{g}^{\wedge}$. 

After assuming that $\tau$ is sufficiently small we can pretend that
\begin{equation}
    K(X,\tau) \approx \exp(B_{0}(\tau/X)), \nonumber
\end{equation}
where $B_{0}(\tau/X)\subseteq \mathfrak{g}$ is a ball of radius $\tau/X$ around $0$. This allows to fix a smoothened characteristic function $\chi$ on $B_0(\tau/X)$. Further, we fix a symbol $a\colon \mathfrak{g}^{\wedge}\to \C$, which is constant on some fixed compact set $\Omega_0\subseteq \mathfrak{g}^{\wedge}$. For this heuristic, we will work with
\begin{equation}
    F_X(g) = X^{\dim(\mathfrak{g})}   \chi(\log(g)) a^{\vee}(X \log(g)). \nonumber
\end{equation}
The presence of $\chi$ ensures that $F_X$ is supported in $K(X,\tau)$, but we have modified it by adding the Fourier transform $a^{\vee}$ of some symbol $a$. This is very close to our choice made in Section~\ref{sec:spectral} below. (In practice, $F_X$ is the self-convolution of such a function, but we ignore this technicality at this point.)

Recall the Jacobian of the $\exp$ map from $\g\supset B_0(\tau)\to G$ is defined by $\d g=j(y)\d y$ for some $j\colon B_0(\tau) \to \R_{>0}$ and satisfies $j(0)=1$. We arrive at (see \S\ref{sec:lie-theory})
\begin{equation}
    \pi(F_X)  = \int_{\mathfrak{g}} \chi(Y) a_{1/X}^{\vee}(y)\pi(\exp(Y))j(yY)\d Y,\nonumber
\end{equation}
where $a_{1/X}(\xi) := a(\xi/X)$. In view of the operator assignment defined in \cite[(2.2)]{nelson2021orbit}, we obtain
\begin{equation}
    \pi(F_X) \approx \textrm{Op}_{1/X}(a;\chi).\nonumber
\end{equation}
Here we have used the approximation $j(y)\approx 1$ as $y\approx 0$. We compute the trace using the Kirillov character formula as in \cite[Theorem~9]{nelson2021orbit} and obtain
\begin{equation}
    \tr(\pi(F_X)) \approx X^{d(\mathcal{O}_{\pi})} \int_{\frac{1}{X}\mathcal{O}_{\pi}} a(\xi) d\omega_{\frac{1}{X}\mathcal{O}_{\pi}}(\xi)\d\xi.\nonumber
\end{equation}
Now, if the re-scaled orbit $\frac{1}{X}\mathcal{O}_{\pi}$ passes through the fixed compact set $\Omega_0$, which is determined by $a$, then the $\xi$-integral can be bounded from below by a constant. Otherwise, we use the decay of $a$ to show that the integral is negligible. Since the location of the orbit $\frac{1}{X}\mathcal{O}_{\pi}$ translates into bounds on $\mu_{\infty}(\pi)$ we obtain
\begin{equation}
    \tr(\pi(F_X)) \approx X^{d(\mathcal{O}_{\pi})} \text{ for }\Vert \mu_{\infty}(\pi)\Vert \ll X \label{eq:lower_trace_heur}
\end{equation}
and that $\tr(\pi(F_X))$ is negligible for $\Vert \mu_{\infty}(\pi)\Vert \gg X^{1+\epsilon}$. Here $d(\calO_\pi) = \frac{n(n-1)}{2}$ is half the (real) dimension of the co-adjoint orbit for generic $\pi$. This reflects the Item \ref{item:invariance} and Item \ref{item:dimension} described above.

\section{Notation and Preliminaries}\label{sec:not}

We will now introduce all the frequently used notations and recall some important preliminaries from the theory of automorphic forms and representations. Most of the notation is taken from \cite{jana2021application}.

We fix $n\in \N$ with $n\geq 4$ and suppress it from the asymptotic notation. In addition, all implied constants are allowed to depend on $\epsilon$, which always represents a sufficiently small positive number that can vary from line to line.

\subsection{Groups, measures and matrices}

Let $G:=\GL_n(\R)$.  We write $B$ for the standard Borel subgroup of $G$ consisting of upper triangular matrices and let $K_{\infty} :=\O_n(\R)$ be the maximal compact subgroup of $G$. We have the decomposition $B=AN$, where $A$ is the maximal split torus consisting of diagonal matrices and $N$ is the unipotent radical of $B$. We denote the  $r\times r$-identity matrix by $\I_r$. The centre of $G$ is denoted by $Z$ and the Weyl group by $W$. The later contains the special element
\begin{equation}
    w_{\ast} := \left(\begin{matrix} & & -1\\ & \I_{n-2} & \\ 1&& \end{matrix}\right). \nonumber
\end{equation}
Given $w\in W$ we define
\begin{equation*}
    N_w := w^ {-1}N^\top w\cap N. 
\end{equation*}
For example, we have
\begin{equation}
    N_{w_{\ast}} = \left\{ \left(\begin{matrix} 1 & \ast & \ast \\  & \I_{n-2} & \ast \\ &&1\end{matrix}\right)\in N \right\}.\nonumber 
\end{equation}

Given a parameter $X>1$ we write
\begin{equation}
    a_X := \diag(X^{1-n},X^{2-n},\ldots, X^{-1},1)\in A.\nonumber
\end{equation}
Note that $Z\subseteq A$. Given $\mathbf{m}=(m_1,\ldots, m_{n-1})\in \N^{n-1}$ we associate the element
\begin{equation}\label{eq:form-of-tilde-m}
    \widetilde{\mathbf{m}} := \diag(m_1\cdots m_{n-1},m_1\cdots m_{n-2},\ldots,m_1,1) \in A.
\end{equation}
Similarly, we define
\begin{equation}\label{eq:form-of-c*}
    c^{\ast} := \diag(1/c_{n-1},c_{n-1}/c_{n-2},\ldots, c_2/c_1,c_1)
\end{equation}
for $c\in \Z_{\neq 0}^{n-1}$. We put $c_0=c_n=1$ for notational convenience. Note that $\det(c^{\ast})=1$, so that it lies in $A\cap \SL_n(\R)$.

We equip $N$ with the Lebesgue measure coming from $\R^{\frac{1}{2}n(n-1)}$ and $A\cong (\R^{\times})^{n}$ with the usual Haar measure. The Haar measure on $K_{\infty}$ is normalized to be a probability measure. Having made these choices we can write the Haar measure on $G$ in Iwasawa coordinates as
\begin{equation}
    \int_{G} f(g)\d g = \int_{N}\int_{A_+}\int_{K_{\infty}} f(nak)\delta(a)^{-1} \d k\d a\d n,\label{eq:measure}
\end{equation}
where $\delta$ is the modular character on $B$ and $A_+:=A\cap(\R_+)^n$. Further, for $a\in A$ we define $\delta_w(a)$ to be the Jacobian of the map $N_w\ni x\mapsto axa^{-1}\in N_w$. An important example is
\begin{equation}\label{eq:value-delta}
    \delta_{w_{\ast}}(\widetilde{\mathbf{m}}) = m^{n-1} \text{ for }\mathbf{m}=(m,1,\ldots,1).
\end{equation}

Given parameters $X>1$ and $0<\tau<1$ we define
\begin{equation}
    K(X,\tau) := \{ g\in G\,\colon\, \| g-\I_n\|_{\infty}<\tau/X\} \text{ and }K(X,\tau)^{\sharp}  := a_X^{-1}K(X,\tau)a_X.\nonumber
\end{equation}
These are not subgroups, but they satisfy certain approximate invariance properties analogous to \cite[Lemma~1.1]{jana2020analytic}.

Given a function $F_X$ with support in $K(X,\tau)$, we also define  
\begin{equation}
    F_X^{\sharp}(g) := F_X(a_Xga_X^{-1}). \label{eq:def_FX}
\end{equation}
One checks that the support of $F_X^{\sharp}(g)$ is in $K(X,\tau)^{\sharp}$.

\subsection{Lie algebras and parameters}\label{sec:lie-theory}

Let $\mathfrak{g}$ denote the Lie algebra $\mathrm{Lie}(G)$ of $G$. Put $\mathfrak{g}_{\C}:=\mathfrak{g}\otimes_{\R} \C$, the complexification of $\g$. Similar notation is used for the Lie algebras $\mathfrak{n}=\textrm{Lie}(N)$, $\mathfrak{a}=\textrm{Lie}(A)$ and $\overline{\mathfrak{n}}=\textrm{Lie}(N^\top)$.

On a suitable neighbourhood of the identity we can also describe the Haar measure from \eqref{eq:measure} on the level of the Lie algebra $\mathfrak{g}$. More precisely, let $B_0(\tau_0)\subseteq \mathfrak{g}$ be a sufficiently small ball around $0$ in $\mathfrak{g}$. Then there is a smooth function $j\colon B_0(\tau_0)\to \R_{>0}$ with $j(0)=1$ such that $\d g=j(Y)\d Y$, where $\d Y$ is the Lebesgue measure on $\mathfrak{g}\cong \R^{n^2}$.

We write $\mathfrak{g}^{\wedge}$ for  the Pontryagin dual of $\mathfrak{g}$ and recall the identification
\begin{equation}
    \mathfrak{g}^{\wedge} = i\mathfrak{g}^{\ast} = \Hom(\mathfrak{g},i\R) \cong i\mathfrak{g}.\nonumber
\end{equation}
The identification is explicitly given by
\begin{equation}
    i\mathfrak{g} \ni \xi \mapsto [x\mapsto \tr(x\xi)]\in \mathfrak{g}^{\wedge}.\nonumber
\end{equation}
The complex dual of $\g_\C$ is denoted by $\mathfrak{g}_{\C}^{\ast}$.

We now set up some notation following \cite[Section~7.1 and~7.2]{nelson2023standard}. We write $[\mathfrak{g}_{\C}^{\ast}]$ for the GIT quotient $\mathfrak{g}_{\C}^{\ast}\parallelslant G$ (according to the co-adjoint action of $G$) and $[\mathfrak{g}^{\wedge}]$ for its real form. We have maps
\begin{equation}
    \mathfrak{g}_{\C}^{\ast} \ni \xi \mapsto[\xi]\in [\mathfrak{g}_{\C}^{\ast}]. \nonumber
\end{equation}
We use the same notation for the corresponding map $\mathfrak{g}^{\wedge}\to [\mathfrak{g}^{\wedge}]$. 
Recall that there is a natural way to scale elements in $ [\mathfrak{g}_{\C}^{\ast}]$ (resp.\ $[\mathfrak{g}^{\wedge}]$) by scalars in $\C$ (resp.\ $\R$). 
Following \cite[Section~7.2.5]{nelson2023standard} we can associate a multiset of complex numbers 
\begin{equation}
    \textrm{eval}(\lambda)=\{\lambda_1,\ldots,\lambda_n\}. \nonumber
\end{equation} 
to $\lambda\in [\mathfrak{g}_{\C}^{\ast}]$.

\subsection{Characters and representations}\label{sec:char-rep}

We define the character $\boldsymbol{\psi}_X\colon N\to S^1$ by
\begin{equation}
    \boldsymbol{\psi}_X(x) = e\left(X\cdot \sum_{i=1}^{n-1}x_{i,i+1}\right) \text{ for }x\in N;\quad e(z):=\exp(2\pi i z). \label{eq:def_char}
\end{equation}
To ease notation we write $\boldsymbol{\psi} = \boldsymbol{\psi}_1$ for the standard choice. Note that
\begin{equation}
    \boldsymbol{\psi}_X(x) = \boldsymbol{\psi}(a_X^{-1}xa_X) \text{ for }x\in N. \nonumber
\end{equation}

We write $\widehat{G}$ for the unitary dual of $G$. Given any irreducible admissible representation $\pi$ of $G$ we associate and infinitesimal parameter $\lambda_{\pi}\in [\mathfrak{g}_{\C}^{\ast}]$. Note that if $\pi\in \widehat{G}$ then $\lambda_{\pi}\in [\mathfrak{g}^{\wedge}]$; see \cite[Section~7.2.5]{nelson2023standard}. We suggestively write $$\mu_{\infty}(\pi) := \textrm{eval}(\lambda_{\pi}) = \{\lambda_{\pi,1},\ldots,\lambda_{\pi,n}\}.$$ This (multi)-set of complex numbers is called the \emph{spectral parameter} (\emph{aka}, the archimedean Satake parameters) of $\pi$. From \cite[Lemma~7.1]{nelson2023standard} we recall that, for $X\geq 1$, the following two properties are equivalent:
\begin{itemize}
    \item $X^{-1}\lambda_{\pi}$ lies in some fixed compact subset of $[\mathfrak{g}_{\C}^{\ast}]$;
    \item $\lambda_{\pi,j}\ll X$ for $1\leq j\leq n$. 
\end{itemize}

We recall the element $\tau(\pi,\boldsymbol{\psi},X)\in \mathfrak{g}^{\wedge}$ from \cite[Definition~7.7]{nelson2023standard}. In particular, this element satisfies
\begin{equation}\label{eq:tau}
    [\tau(\pi,\boldsymbol{\psi},X)] = X^{-1}\lambda_{\pi}. \nonumber 
\end{equation}
From \cite[Lemma~7.8]{nelson2023standard} we recall that, if $X^{-1}\lambda_{\pi}$ belongs to some fixed compact subset of $[\mathfrak{g}^{\wedge}]$, then $\tau(\pi,\boldsymbol{\psi},X)$ belongs to a fixed compact subset of $\mathfrak{g}^{\wedge}_{\textrm{reg}}.$ Thus, our discussion above allows us to find a fixed compact set
\begin{equation}
    \Omega_0\subseteq \mathfrak{g}^{\wedge}_{\textrm{reg}} \label{eq:Omega_0}
\end{equation}
such that $\Vert \mu_{\infty}(\pi)\Vert \leq X$ if and only if $\tau(\pi,\boldsymbol{\psi},X)\in \Omega_0$.

We will now adopt some notation from \cite[Section~8.1]{nelson2023standard} and introduce the Whittaker models for the generic representations $\pi\in \widehat{G}$. We write $\pi^{\infty}$ for the smooth vectors in $\pi$ and $\pi^{-\infty}$ for the continuos dual of $\pi^{\infty}$. We say that $\pi$ is \emph{generic} if there is a non-trivial vector $\Theta_{\pi,\boldsymbol{\psi}_X}\in \pi^{-\infty}$ such that
\begin{equation}\label{eq:def-dist-vec}
    \pi(x)\Theta_{\pi,\boldsymbol{\psi}_X} = \boldsymbol{\psi}_X(x)\Theta_{\pi,\boldsymbol{\psi}_X} \text{ for all }x\in N.
\end{equation}
Note that if $\Theta_{\pi,\boldsymbol{\psi}_X}$ exists for some $X$ then it exists for all $X$. Furthermore, if $\pi$ is generic, then $\Theta_{\pi,\boldsymbol{\psi}_X}$ is unique up to scalars. We define the $\boldsymbol{\psi}_X$-Whittaker model of $\pi$ by
\begin{equation}
    \mathcal{W}(\pi,\boldsymbol{\psi}_X) = \{ g\mapsto \langle \pi(g)v,\Theta_{\pi,\boldsymbol{\psi}_X}\rangle\mid v\in\pi^\infty\}.\nonumber 
\end{equation}
The inner product in the Whittaker model is given by
\begin{equation}
    \langle W_1,W_2\rangle_{\mathcal{W}(\pi,\boldsymbol{\psi}_X)} := \int_{(N(\R)\cap \GL_{n-1}(\R))\backslash \GL_{n-1}(\R)} W_1\left[\left(\begin{matrix} h & 0 \\ 0 & 1\end{matrix}\right)\right] \overline{W_2\left[\left(\begin{matrix} h & 0 \\ 0 & 1\end{matrix}\right)\right]}\d h.\nonumber
\end{equation}
    
\begin{remark}\label{rm:changing_char}
There is a natural map
\begin{equation}
    \mathcal{W}(\pi,\boldsymbol{\psi}) \to \mathcal{W}(\pi,\boldsymbol{\psi}_X) ,\, W\mapsto [g\mapsto X^{-\frac{1}{12}n(n-1)(n-2)}W(a_X^{-1}g)]. \nonumber
\end{equation}
Note that this is normalized so that the map is isometric with respect to the inner product in the Whittaker model.    
\end{remark}

\subsection{The Bessel distribution}

Given an irreducible generic unitary representation $\pi$ of $G$ we define the Bessel distribution as follows. Let $\mathcal{B}_{\boldsymbol{\psi}_X}(\pi)$ denote a orthonormal basis for $\mathcal{W}(\pi,\boldsymbol{\psi}_X)$ consisting of smooth vectors. Then, for $F\in \mathcal{C}_c^{\infty}(G)$ we set
\begin{equation}
    J_{\pi,\boldsymbol{\psi}_X}(F) := \sum_{W\in \mathcal{B}_{\boldsymbol{\psi}_X}(\pi)} \pi(F)W(1)\overline{W(1)}.\nonumber
\end{equation}
Recall \cite[Lemma~2.2]{jana2021application}, which states that
\begin{equation}
    J_{\pi,\boldsymbol{\psi}_X}(f\ast f^{\vee}) = \sum_{W\in \mathcal{B}_{\boldsymbol{\psi}_X}(\pi)} \vert \pi(f)W(1)\vert^2.\nonumber
\end{equation}
Here $\ast$ denotes the convolution and $f^\vee(g):=\overline{f\left(g^{-1}\right)}$
This is important in order to guarantee positivity of the Bessel distributions in practice.

\begin{remark}
Given a function $F:=f\ast f^{\vee}$ we can define $\widetilde{F}(g) := F\left(a_X^{-1}ga_X\right)= \widetilde{f}\ast \widetilde{f}^{\vee}$ with $\widetilde{f} := L_{a_X}f:=f\left(a_X^{-1}\cdot\right)$. In view of Remark~\ref{rm:changing_char} we can compute
\begin{align}
    J_{\pi,\boldsymbol{\psi}}(F) &= \sum_{W\in \mathcal{B}_{\boldsymbol{\psi}}(\pi)}  \left\vert \pi\left(L_{a_X^{-1}}\widetilde{f}\right)W(1)\right\vert^2 = \sum_{W\in \mathcal{B}_{\boldsymbol{\psi}}(\pi)}\left\vert \pi\left(\widetilde{f}\right)W(a_X^{-1})\right\vert^2 \nonumber\\
    &= X^{\frac{1}{6}n(n-1)(n-2)}\sum_{W\in \mathcal{B}_{\boldsymbol{\psi}_X}(\pi)}\left\vert \pi\left(\widetilde{f}\right)W(1)\right\vert^2 = X^{\frac{1}{6}n(n-1)(n-2)}J_{\pi,\boldsymbol{\psi}_X}(\widetilde{F}).\nonumber  
\end{align}
In particular, applying this to $F=F_X^{\sharp}$ and $\widetilde{F}=F_X$ (see \eqref{eq:def_FX}) we observe that
\begin{equation}
    J_{\pi,\boldsymbol{\psi}}(F_X^{\sharp}) =  X^{\frac{1}{6}n(n-1)(n-2)} J_{\pi,\boldsymbol{\psi}_X}(F_X) = X^{\frac{1}{6}n(n-1)(n-2)} \sum_{W\in \mathcal{B}_{\boldsymbol{\psi}_X}(\pi)}\vert \pi(f_X)W(1)\vert^2.\label{eq:conversion}
\end{equation}
\end{remark}

\subsection{Automorphic forms and representations}

We let $\Gamma:=\GL_n(\Z)$ and write $\mathbb{X} := Z\Gamma\backslash G$ which we equip with the $G$-invariant probability measure. We are actually interested in those representations $\pi\in \widehat{G}$ that appear in the spectral decomposition of $L^2(\mathbb{X})$. We write $\widehat{\mathbb{X}}_{\textrm{disc}}$ for the class of all discrete automorphic representations and $\widehat{\mathbb{X}}_{\textrm{cusp}}$ for the subclass consisting of cuspidal representations.

Given $\pi\in \widehat{\mathbb{X}}_{\textrm{cusp}}$ we define the Whittaker period by
\begin{equation}
    \mathcal{W}_{v}(g) := \int_{N\cap\Gamma\backslash N} \pi(xg)v\, \overline{\boldsymbol{\psi}(x)}\d x, \quad v\in \pi.\nonumber
\end{equation}
By \cite{lapid2015whittaker} there is a number $\ell(\pi)\in \R_+$ such that
\begin{equation}
    \Vert v\Vert_{L^2(\mathbb{X})}^2 = \ell(\pi)\cdot \Vert \mathcal{W}_{v}\Vert_{\mathcal{W}(\pi,\boldsymbol{\psi})}^2.\nonumber
\end{equation}
Moreover, we have $\ell(\pi)\asymp L(1,\pi,\textrm{Ad})$. For $\mathbf{m}\in \N^{n-1}$ we follow \cite[(2.5)]{jana2021application} and define the $\boldsymbol{m}$'th Fourier coefficient by
\begin{equation}
    \int_{N\cap\Gamma\backslash N} \pi(xg)v\, \overline{\boldsymbol{\psi}(\widetilde{\mathbf{m}}x\widetilde{\mathbf{m}}^{-1})}\d x = \frac{\lambda_{\pi}(\mathbf{m})}{\delta^{\frac{1}{2}}(\widetilde{\mathbf{m}})}\mathcal{W}_v(\widetilde{\mathbf{m}}g).\nonumber
\end{equation}
We obviously have $\lambda_{\pi}((1,\ldots,1)) = 1$. Furthermore, the numbers $\lambda_{\pi}$ are certain Hecke eigenvalues of $\pi$ and assuming the generalized Ramanujan conjecture we have $\lambda_{\pi}(\mathbf{m}) \ll \vert m_1\cdots m_{n-1}\vert^{\epsilon}$.

Now let us define the family
\begin{equation}
    \Omega_{\textrm{cusp}}(X) := \{\pi\in \widehat{\mathbb{X}}_{\textrm{cusp}}\colon  \Vert \mu_{\infty}(\pi)\Vert\leq X\}.\label{eq:def_family}
\end{equation}
Given a prime $p$ we can attach a set of $p$-adic Satake parameter $\mu_p(\pi)$. This is a (multi)-set 
\begin{equation}
    \mu_p(\pi) = \{\lambda_p(\pi)_1,\ldots,\lambda_p(\pi)_{n}\} \nonumber
\end{equation}
of complex numbers. These are uniquely determined by the numbers $\lambda_{\pi}((m_1,\ldots,m_{n-1)}))$, where $m_1,\ldots,m_{n-1}$ are powers of $p$. We define
\begin{equation*}
    \sigma_p(\pi) = \max_{i=1,\ldots,n-1} \vert \Re(\lambda_p(\pi)_i)\vert. \label{def:sigma_p}
\end{equation*}
This is a measure for the non-temperedness of $\pi$ at $p$. In particular, under the generalized Ramanujan conjecture we expect that $\sigma_p(\pi)=0$ for all $\pi\in \widehat{\mathbb{X}}_{\textrm{cusp}}$. We record the following facts:
\begin{itemize}
    \item If $\pi$ is the trivial representation, then $\sigma_p(\pi)=\frac{n-1}{2}$ for all $p$.
    \item If $\pi\in \widehat{\mathbb{X}}_{\textrm{cusp}},$ then there is $\delta_n>0$ such that $\sigma_p(\pi)\leq \frac{1}{2}-\delta_n$; see \cite{blomer2011ramanujan}.
\end{itemize}

The notions introduced above can be extended to general (\emph{i.e.}, not necessarily discrete) automorphic representations. These make up the class $\widehat{\mathbb{X}}$ which is equipped with an automorphic Plancherel measure $\mu_{\mathrm{aut}}$. In particular, given $F\in \mathcal{C}_c^{\infty}(\X)$ we can write the spectral decomposition of the \emph{right-translation} operator $R(F)\colon L^2(\mathbb{X})\to L^2(\mathbb{X})$ as
\begin{equation}
    R(F) = \int_{\widehat{\mathbb{X}}} \pi(F)\d\mu_{\mathrm{aut}}(\pi). \nonumber
\end{equation}
We denote $\widehat{\mathbb{X}}_{\textrm{gen}}$ to be the subspace of generic representations. Note that $\widehat{\mathbb{X}}_\gen\cap\widehat{\mathbb{X}}_{\textrm{disc}}=\widehat{\mathbb{X}}_{\textrm{cusp}}$.
Finally, the measure $\mu_{\textrm{aut}}$ restricted to $\widehat{\mathbb{X}}_{\textrm{disc}}$ is discrete.

\subsection{The (pre)-Kuznetsov formula}

In this section, we recall the (pre)-Kuznetsov formula for $\GL_n$, which is a special relative trace formula with many important applications in analytic number theory. We first need to introduce some further notation.

Given $\mathbf{m},\mathbf{n}\in \N^{n-1}$ and $w\in W$ we call a modulus $c\in \Z_{\neq0}^{n-1}$ \emph{admissible} if
\begin{equation*}
    \boldsymbol{\psi}\left(\widetilde{\mathbf{m}}c^{\ast}wxw^{-1}(c^{\ast})^{-1}\widetilde{\mathbf{m}}^{-1}\right) = \boldsymbol{\psi}(\widetilde{\mathbf{n}}x\widetilde{\mathbf{n}}^{-1}) \text{ for all }x\in w^{-1}Nw\cap N.\label{eq:def_ad}
\end{equation*}
The admissibility of $c$ of course depends on $\mathbf{m},\mathbf{n}$ and $w$, but this dependence will not be emphasised. One can check that the admissibility condition can only be satisfied if $w\in W$ is of the form
\begin{equation}\label{eq:admissible-weyl}
    \left(\begin{matrix} & & \I_{d_1} \\ & \iddots & \\ \I_{d_k}\end{matrix}\right)\text{ with }d_1+\ldots+d_k=n;
\end{equation}
see \cite[Lemma~2.1]{jana2021application} for example.
For example, if $w$ is the long Weyl element (\emph{i.e.}, $d_1=\cdots = d_n=1$), then $w^{-1}Nw\cap N = \{\I_n\}$, then all $c$ are admissible.

For non admissible $c$ we set
\begin{equation}
    S_w(\mathbf{m},\mathbf{n};c)=0. \nonumber
\end{equation}
Otherwise these quantities are Kloosterman sums defined by
\begin{equation*}
    S_w(\mathbf{m},\mathbf{n};c) = \sum_{\substack{\gamma\in N\cap\Gamma \backslash (\Gamma\cap BwN_w/\Gamma\cap N_w\\ \gamma=xc^{\ast}wy}} \boldsymbol{\psi}(\widetilde{\mathbf{m}}x\widetilde{\mathbf{m}}^{-1})\boldsymbol{\psi}(\widetilde{\mathbf{n}}y\widetilde{\mathbf{n}}^{-1}).
\end{equation*}
We easily compute that
\begin{equation}
    S_1(\mathbf{m},\mathbf{n};c)=\delta_{\mathbf{m}=\mathbf{n}}\delta_{c=(1,\ldots,1)}.\nonumber
\end{equation}
In general the Kloosterman sums are highly non-trivial objects. We have the trivial bound
\begin{equation}\label{eq:triv-bound-klooster}
    S_{w}(\mathbf{m},\mathbf{m};c) \ll_{\epsilon} \vert c_1\cdots c_{n-1}\vert^{1+\epsilon},
\end{equation} 
which is sufficient for our purposes. However, due to recent results in \cite{linn2024kloosterman} non-trivial bounds for all Kloosterman sums are available.

Finally, we need to define the archimedean version of the Kloosterman sums. These will be certain (twisted relative) orbital integrals. Given a test function $F\in \mathcal{C}_c^{\infty}(Z\backslash G)$ we associate $W_F\in \mathcal{C}_c^{\infty}(ZN\backslash G,\boldsymbol{\psi})$ given by
\begin{equation}\label{eq:def-orb-int-1}
    W_F(g) :=\int_{N} F(xg)\overline{\boldsymbol{\psi}(x)}\d x.
\end{equation}
With this at hand we define the (twisted relative) orbital integrals
\begin{equation}\label{eq:def-orb-int-gen}
    \mathcal{O}_{\boldsymbol{\psi}}(F;g) = \int_{N_w}W_F(gy)\overline{\boldsymbol{\psi}(y)}\d y,
\end{equation}
for $g\in BwN_{w}$. This integral is absolutely convergent; see for example \cite[p.\ 496]{getz2024intro}. Note that the definition of $\mathcal{O}_{\boldsymbol{\psi}}(F;g)$ depends on the Bruhat cell that contains $g$.

We record the following version of the Kuznetsov formula taken from \cite[Proposition~2.1]{jana2021application}. See also \cite[Lemma~5.1]{assing2024principal} for a related result in classical formulation.

\begin{prop}\label{prop:Kuznetsov}
Let $F\in \mathcal{C}_c^{\infty}(Z\backslash G)$ and $\mathbf{m}\in \N^{n-1}$. Then we have
\begin{equation}
    \int_{\widehat{\mathbb{X}}_{\mathrm{gen}}} \vert \lambda_{\pi}(\mathbf{m})\vert^2J_{\pi,\boldsymbol{\psi}}(\overline{F})\frac{\d\mu_{\mathrm{aut}}(\pi)}{\ell(\pi)} = \mathop{\sum\sum}_{\genfrac{}{}{0pt}{}{w\in W\,c\in \Z_{\neq 0}^{n-1}}{\mathrm{admisible}}}\frac{S_w(\mathbf{m},\mathbf{m};c)}{\delta_w(\widetilde{\mathbf{m}})} \mathcal{O}_{\boldsymbol{\psi}}(F,\widetilde{\mathbf{m}}c^{\ast}w\widetilde{\mathbf{m}}^{-1}).\nonumber
\end{equation}
\end{prop}

\section{Spectral Estimate}\label{sec:spectral}

Before we state the main result of this section we will introduce the relevant test function. We will borrow quite a few notations, definitions, and results from \cite{nelson2023standard}. For the ease of translation we note that $\mathrm{h}$ in our case is $X^{-1}$. 

We fix a sufficiently small $\delta>0$. Also, we fix a $0$-\emph{admissible cut-off} (see \cite[\S 13.4.1]{nelson2023standard} for the definition) $\chi$ on $\g$ satisfying
\begin{itemize}
    \item $\chi$ is supported on a ball of radius $\O\left(X^{-1+\delta}\right)$ in $\mathfrak{g}$;
    \item $\chi\equiv 1$ on a ball of radius  $\O\left(X^{-1}\right)$; and
    \item $\|\partial_\alpha\chi\|_\infty\ll_\alpha X^{\O(1)}$ for any multi-indices $\alpha$.
\end{itemize}
For any symbol $a\in S_{0}^\infty(\g^\wedge)$ (see \cite[\S 13.3]{nelson2023standard} for the definition) we define its Fourier transform by
$$a^\vee(x) : =\int_{\g^\wedge}a(\xi)\exp(-\tr (x\xi))\d \xi,\quad x\in\g.$$
For any cut-off $\chi$ we recall the operator assignment from \cite[Section~13.6.1]{nelson2023standard}:
\begin{equation}
    \Op(a:\chi):=\int_\g a^\vee(x)\chi(x)\pi(\exp(x))\d x. \label{eq:def_OP_1}
\end{equation}
We also denote
$$a_{X^{-1}}(\xi):=a(X^{-1}\xi),\quad \Op_{X^{-1}}(a:\chi):=\Op(a_{X^{-1}}:\chi).$$
Letting $a$ run through a sequence of symbols in $S_0^\infty(\g^\wedge)$ that are constantly $1$ on increasing neighbourhoods of the identity, so that $a^\vee$ tends to the Dirac-$\delta$ at $x=0$, we immediately see that
\begin{equation}\label{eq:constant-symbol}
    \Op_{X^{-1}}(1:\chi)=\lim\int_\g a^\vee(x)\chi(X^{-1}x)\pi(\exp(X^{-1}x))\d x=\mathrm{Id};
\end{equation}
where $1$ on the left-hand side above denotes the symbol in $S_0^\infty(\g^\wedge)$ that is constantly $1$.

Recall the fixed compact set $\Omega_0\subseteq \mathfrak{g}^{\wedge}_{\textrm{reg}}$ from \eqref{eq:Omega_0} and fix $a\in S_{0}^{-\infty}(\g^\wedge)$ that is identically $1$ on $X^{\delta'}\Omega_0$ for $0<\delta'<\delta$.
We define $f_X\in C_c^\infty(G)$ by the formula
$$\int_G f_X(g)\pi(g)\d g:=\int_\g\chi(x)a_{X^{-1}}^{\vee}(x)\pi(\exp(x))\d x,$$
that is, $\Op_{X^{-1}}(a:\chi)=\pi(f_X)$ (in the notation of \cite[\S 13.6.1]{nelson2023standard} we can write $f_X=\widetilde{\Op}_{X^{-1}}(a:\chi)$).

We now record some required estimates and support of the chosen test function. The arguments below are essentially the same as in \cite[\S 14.9]{nelson2023spectral}. From the support condition on $\chi$ it follows that $f_X$ is supported on $K\left(X^{1-\delta''},\tau'\right)$ for some sufficiently small $\tau',\delta''>0$. Moreover, integrating by parts we see that
$a^\vee_{X^{-1}}(x)$ decays rapidly once $|x|\gg_\epsilon X^{1-\delta'+\epsilon}$, and consequently,
$$\|f_X\|_{L^1(G)}\,\ll \left\|a^\vee_{X^{-1}}\right\|_{L^1(\g)}\ll 1,\quad \|f_X\|_{L^\infty(G)}\ll X^{(1-\delta')\dim\g}.$$

Finally, define $F_X:=f_X\ast {f_X}^\vee$. From the properties of $f_X$ recorded above, we obtain that for some sufficiently small but fixed $\tau>0$ we have
\begin{equation}\label{eq:support-F}
     \textrm{supp}(F_X)\subseteq K\left(X^{1-\epsilon},\tau\right).
\end{equation}
Moreover, we also note that
$$\Vert F_X\Vert_{\infty}\,\le \|f_X\|_1\|f_X\|_\infty\ll X^{n^2}.$$
Thus we obtain
\begin{equation}\label{eq:supnorm-average}
    \int_{Z} F_X(zg)\d z \ll X^{n^2-1-\epsilon},
\end{equation}
which follows from the fact that $$\vol\left\{z\mid zg\in K\left(X^{1-\epsilon},\tau\right)\right\}\ll_\tau X^{-1+\epsilon}$$
uniformly in $g$.

We recall $F_X^\sharp$ from \eqref{eq:def_FX} where $F_X$ is as above.
\begin{prop}\label{prop:spectral_lower_bound}
Let $X>1$ be sufficiently large and let $\pi$ be a generic unitary, irreducible representation of $G$ with trivial central character. If $\Vert \mu_\infty(\pi)\Vert\, \leq X$ then we have
\begin{equation}
    J_{F_X^{\sharp},\boldsymbol{\psi}}(\pi) \gg_{\epsilon} X^{\frac{1}{6}n(n-1)(n-2)+\frac{1}{2}n(n-1)-\epsilon}.\nonumber
\end{equation}
\end{prop}

For the proof we will need some more preparation. We recall the definition of a \emph{module} from \cite[Definition 14.6]{nelson2023standard}. An important example is the space $\mathcal{V}$ given in \cite[Example 14.8]{nelson2023standard}, which is a $S_{0}^\infty(\g^\wedge)$-module. Finally,  we recall the notion of localization from \cite[Definition 14.2]{nelson2023standard}. The following important result can be extracted from \cite{nelson2023standard}.

\begin{lemma}\label{lm:existence_W}
Let $\pi$ be a generic unitary, irreducible representation of $G$ with trivial central character with $\Vert \mu_\infty(\pi)\Vert \,\leq X$. Then there is  $W\in\calW(\pi,\boldsymbol{\psi}_X)$ that is $(\g,\delta)$-localized at $\tau(\pi,\boldsymbol{\psi},X)$ inside $\mathcal{V}$ satisfying $$W(1)\gg X^{\frac{n(n-1)}{4}(1-\delta)} \text{ and } \|W\|^2\,\ll 1.$$
\end{lemma}

\begin{proof}
We recall $\mathfrak{M}(\pi,\boldsymbol{\psi},X,\delta)\subset \calW(\pi,\boldsymbol{\psi}_X)$ from \cite[Definition 8.4]{nelson2023standard}. In what follows, $\overline{B_H}$ denotes (following \cite{nelson2023standard}) the Borel subgroup of lower triangular matrices in $\GL_{n-1}(\R)$. Any element $W'\in\mathfrak{M}(\pi,\boldsymbol{\psi},X,\delta)$ can be realized by invoking that
$$\gamma:=W'\vert_{\overline{B_H}}\,\in\mathfrak{C}(\overline{B_H},0,X,\boldsymbol{\psi})$$
where the latter object is defined in \cite[Definition 6.35]{nelson2023standard}. In particular, from \cite[Example 6.41]{nelson2023standard}, we have:
\begin{itemize}
    \item $\gamma$ is supported on $1+\O\left(X^{-1+\delta_+}\right)$ for some fixed $\delta_+>\delta$;
    \item We have
    $$\partial^\alpha\gamma(x) \ll_{\alpha,A} X^{(1-\delta)(\dim \overline{B_H}+|\alpha|)}\left(1+X^{1-\delta}\dist(x,1)\right)^{-A}$$
    for any multi-index $\alpha\in\Z_{\ge 0}^{\dim\overline{B_H}}$.
\end{itemize}
We fix a $\gamma\in\mathfrak{C}(\overline{B_H},0,X,\boldsymbol{\psi})$, consequently $W'\in\mathfrak{M}(\pi,\boldsymbol{\psi},X,\delta)\subset \calW(\pi,\boldsymbol{\psi}_X)$, satisfying the above properties, along with,
\begin{equation*}\label{eq:l-infty-condition}
W'(1)=\gamma(1) = \|\gamma\|_\infty\, \asymp X^{(1-\delta)\dim \overline{B_H}}.
\end{equation*}
Finally, we recall from \cite[Lemma 8.5]{nelson2023standard} that such $W$ satisfies
\begin{equation*}\label{eq-l-2-condition}
    \|W'\|^2 \, \ll X^{(1-\delta)\dim \overline{B_H}}.
\end{equation*}
We recall from \cite[Corollary 17.11]{nelson2023standard} that
$X^{-(1-\delta)\frac{\dim\overline{B_H}}{2}}W'$ for each $W'\in\mathfrak{M}(\pi,\boldsymbol{\psi},X,\delta)$ is $(\g,\delta)$-localized at $\tau(\pi,\boldsymbol{\psi},X)$ inside $\mathcal{V}$. Thus $W=X^{-(1-\delta)\frac{\dim\overline{B_H}}{2}}W'$ is as desired.
\end{proof}

We can now complete the proof of Proposition~\ref{prop:spectral_lower_bound}. Let $\pi$ be as in the statement and take $W\in\calW(\pi,\boldsymbol{\psi}_X)$ as constructed in Lemma~\ref{lm:existence_W}. Then, from \eqref{eq:conversion} we obtain the lower bound
\begin{equation}
    J_{F_X^{\sharp},\boldsymbol{\psi}}(\pi) \gg_{\epsilon} X^{\frac{1}{6}n(n-1)(n-2)} \vert \pi(f_X)W(1)\vert^2.\nonumber
\end{equation}
Thus it is sufficient to show that
\begin{equation}
    \pi(f_X)W(1) = (1+o(1))W(1)\gg X^{\frac{n(n-1)}{4}(1-\delta)}.\nonumber
\end{equation}
To do so we will use the specific choice of $f_X$ as well as the localization properties of $W$. 

By the choice of $a$ in the definition of $f_X$ we find that the symbol $a-1\in S^\infty_{0}(\g^\wedge)$ vanishes identically on $\tau(\pi,\boldsymbol{\psi},X)+\o\left(1\right)$. Thus applying \cite[Theorem 14.12 (i)]{nelson2023standard} we obtain
$$\Op_{X^{-1}}(a-1:\chi)W\in X^{-A}\,\mathcal{V},\quad\text{for all } A>0.$$
Using \eqref{eq:constant-symbol} we write
$$\Op_{X^{-1}}(a:\chi)W(1)-W(1) \ll_A X^{-A} \sup_{W'\in\mathcal{V}} |W'(1)|.$$
Using a Sobolev inequality (see \emph{e.g.}, \cite[Proof of Lemma 16.1]{nelson2023standard}) and definition of $\mathcal{V}$ in \cite[Example 14.8]{nelson2023standard} we obtain that the right-hand side above is $\O_A\left(X^{-A}\right)$. Expanding the definition of $\Op$ (see \eqref{eq:def_OP_1}) we find that
\begin{equation}
    \pi(f_X)W(1) = \Op_{X^{-1}}(a:\chi)W(1) = W(1)+\O_A(X^{-A}). \nonumber
\end{equation}
This completes the proof.

\section{Orbital Integrals}\label{sec:orbit}

Throughout this section we assume $n\ge 4$. We fix a test function $F$ with support in $ZK(X,\tau)^{\sharp}$ for $\tau>0$ sufficiently small and $X\geq 1$ a large parameter. The main result, Proposition~\ref{prop:geometric_bound} below, is a general estimate for the geometric side of the Kuznetsov formula, as in Proposition \ref{prop:Kuznetsov}, for this test function $F$

\begin{prop}\label{prop:geometric_bound}
Fix $n\geq 4$. Let $X>1$ be sufficiently large and let $\tau>0$ be sufficiently small (but fixed). Suppose that $F\in \mathcal{C}_c^{\infty}(G)$ is a test function supported in $K(X,\tau)$. Finally, let $\mathbf{m}=(m,1,\ldots,1)\in\N^{n-1}$ with $m \leq X^{n+2}$. Then we have
\begin{equation}
    \mathop{\sum\sum}_{\genfrac{}{}{0pt}{}{w\in W\,c\in \Z_{\neq 0}^{n-1}}{\mathrm{admisible}}} \frac{S_w(\mathbf{m},\mathbf{m};c)}{\delta_w(\widetilde{\mathbf{m}})} \mathcal{O}_{\boldsymbol{\psi}}(F,\widetilde{\mathbf{m}}c^{\ast}w\widetilde{\mathbf{m}}^{-1}) \ll_{\epsilon} \Vert F\Vert_{\infty} X^{\frac{1}{6}n(n-1)(n-2)+\epsilon}.\nonumber
\end{equation}
\end{prop}
This estimate will be obtained by estimating everything trivially, using only the support and the $L^{\infty}$-norm of $F$.

The following Lemma is inspired by \cite[Lemma~3.1]{assing2024density} and \cite[Lemma~3.2]{jana2021application}. 

\begin{lemma}\label{lm:weyl_elem}
Let $w\in W\setminus \{1,w_{\ast}\}$ be of the form given in \eqref{eq:admissible-weyl}
and 
$\mathbf{m}:=(m,1,\ldots,1)\in \N^{n-1}$. If $\widetilde{\mathbf{m}}xc^{\ast}wy\widetilde{\mathbf{m}}^{-1} \in K(X,\tau)^{\sharp}$ for some $x\in N$ and some $y\in N_w$ then there exists $j\in \{1,\ldots, n-1\}$ such that $$\vert c_j\vert\, < \frac{m}{X^{n+2}}.$$
\end{lemma}

\begin{proof}
Keeping the same notation as in \eqref{eq:admissible-weyl} we note that $1\le d_k<n$, since $w\neq 1$. We first consider the case $1<d_k< n$. As in the proof of \cite[Lemma~4.2]{assing2024principal} we obtain
\begin{equation}
    xc^{\ast}wy = \left(\begin{matrix} \text{\huge $\ast$} & \cdots & \cdots & \text{\huge $\ast$}  \\ \frac{c_{d_k}}{c_{d_k-1}} & \ast & \ast &\vdots \\ & \ddots & \ast &\vdots \\ & & c_1 & \text{\huge $\ast$} \end{matrix}\right) \in \widetilde{\mathbf{m}}^{-1}K(X,\tau)^{\sharp}\widetilde{\mathbf{m}}
\end{equation}
For $1\leq i\leq d_k$ we have
\begin{equation*}
    \left\vert \frac{c_i}{c_{i-1}}\right\vert \,\leq \frac{\tau}{X^{n-d_k+1}}\left(1+(m-1)\delta_{i=1}\right)
\end{equation*}
Consequently, we obtain 
\begin{equation*}
    \vert c_{d_k}\vert \, \leq \frac{\tau^{d_k}}{X^{d_k(n-d_k+1)}}m.
\end{equation*}
We conclude this case using that for $2\leq d_k\leq n-1$ and $n\geq 4$ we have $d_k(n-d_k+1)\geq n+2$. 

Next we consider $d_k=1 \neq d_1$. Here we define the map $\iota\colon g\mapsto w_l g^{-\top} w_l^{-1}$, where $w_l$ is the long Weyl element.  Note that $\iota(K(X,\tau)^{\sharp}) \subseteq K(X,\tau')^{\sharp}$ for some sufficiently small $\tau'>0$.
We set $\widetilde{\mathbf{n}} := w_l\widetilde{\mathbf{m}}w_l^{-1}$ where $\mathbf{n}=(1,\dots,1,m)$.
Note that $\iota(w)=w^{-1}=\left(\begin{matrix} & & \I_{d_k} \\ & \iddots & \\ \I_{d_1} & & \end{matrix}\right)$ and $y':=\iota(y)\in N_{w^{-1}}$ if and only if $y\in N_w$. Thus for $x':=\iota(x)\in N$ and $\widetilde{c}:=(c_{n-1},\dots,c_1)$ we obtain
\begin{equation}
    \iota(xc^{\ast}wy) = x'\widetilde{c}^\ast w^{-1} y' \in \widetilde{\mathbf{n}} K(X,\tau')^{\sharp}\widetilde{\mathbf{n}}^{-1}
    \nonumber
\end{equation}
We apply a similar argument as above and the result follows.

Finally, we consider $d_k=d_1=1$. Since $w\neq w_{\ast}$, we have $d_{k-1}<n-2$. Here we compute
\begin{equation*}
    c^{\ast}w = \left(\begin{matrix} & & & & \text{\huge $\ast$}\\ & \frac{c_{d_{k-1}+1}}{c_{d_{k-1}}} & & & \\ & & \ddots & & \\ & & & \frac{c_2}{c_1} & \\ c_1 & & & & \end{matrix}\right).\nonumber
\end{equation*}
First, we conclude that
\begin{equation*}
    \vert c_1\vert\, \leq \frac{\tau}{X^{n}}m
\end{equation*}
Arguing as in the proof of \cite[Lemma~4.2]{assing2024principal} we conclude that
\begin{equation}
    \vert c_2\vert\, \leq \tau^2\frac{m}{X^{2n-d_{k-1}-1}} <\frac{m}{X^{n+2}}
\end{equation}
since $d_{k-1}+1\le n-2$ and $\tau$ sufficiently small. This completes the proof.
\end{proof}

Recall the definition of the orbital integral $\mathcal{O}_{\boldsymbol{\psi}}$ from \eqref{eq:def-orb-int-gen}. The following corollary immediately follows from Lemma~\ref{lm:weyl_elem}.

\begin{cor}\label{cor:support_geo}
Let $n,F,\widetilde{\mathbf{m}}$ be as in Proposition \ref{prop:geometric_bound} and $w$ be as in Lemma \ref{lm:weyl_elem}. Then we have
\begin{equation}
    \mathcal{O}_{\boldsymbol{\psi}}(F;\widetilde{\mathbf{m}}c^{\ast}w\widetilde{\mathbf{m}}^{-1})=0\nonumber
\end{equation}
for all $c\in\Z_{\neq 0}^{n-1}$.
\end{cor}

\begin{proof}
    If for some $c$ we have $\widetilde{\mathbf{m}}xc^{\ast}wy\widetilde{\mathbf{m}}^{-1} \in \supp(F) \subset ZK(X,\tau)^{\sharp}$
    then there is a $z\in Z$ such that $z\widetilde{\mathbf{m}}xc^{\ast}wy\widetilde{\mathbf{m}}^{-1} \in K(X,\tau)^{\sharp}$. Checking determinants we conclude $z \asymp 1+\O\left(X^{-1}\right)$. The claim now follows from Lemma \ref{lm:weyl_elem} after possibly modifying $\tau$.
\end{proof}

\begin{remark}\label{rem:w-ast-c-bound}
The same arguments in Lemma \ref{lm:weyl_elem} imply that $\mathcal{O}_{\boldsymbol{\psi}}(F;\widetilde{\mathbf{m}}c^{\ast}w_{\ast}\widetilde{\mathbf{m}}^{-1})=0$ unless $c_1 <\frac{m}{X^n}$.
\end{remark}

Recall the definition of $W_F$ from \eqref{eq:def-orb-int-1}. Note that for $a\in A$ we have
\begin{equation}
    \mathcal{O}_{\boldsymbol{\psi}}(F;a) =W_F(a).\nonumber
\end{equation}
This is because $N_w=\{1\}$ precisely when $w=1$. We record the following trivial estimate, which turns out to be sharp in many situations.

\begin{lemma}\label{lm:trivial_identity}
Let $F$ be as in Proposition \ref{prop:geometric_bound}. We have
\begin{equation}
     W_{F}(1) \ll X^{\frac{1}{6}n(n-1)(n-2)}\Vert F\Vert_{\infty}.\nonumber
\end{equation}
\end{lemma}

\begin{proof}
Clearly, we have
\begin{equation}
    \vert W_F(1)\vert\, \leq \vol(\supp(F)\cap N)\Vert F\Vert_{\infty}. \nonumber
\end{equation}
We are done after noting that $\vol(K(X,\tau)^{\sharp}\cap N) \ll X^{\frac{1}{6}n(n-1)(n-2)}$.
\end{proof}

We now turn towards the slightly exceptional element $w_{\ast}$. First, recall that by \cite[Theorem~4.4]{assing2024principal} we have $S_{w_{\ast}}(\mathbf{m},\mathbf{n};c)=0$ unless\footnote{Technically speaking this is up to signs, but this is irrelevant for our discussion.}
\begin{equation}
    c=(r,rs,\ldots, rs^{n-2}) \text{ or }c=(rs^{n-2},\ldots,rs,r), \label{eq:admis_c}
\end{equation}
where $\mathbf{m}$ and $\mathbf{n}$ are of the shape $(\ast,1,\ldots,1,\ast)$. Arguing as in the proof \cite[Lemma~3.4]{assing2024density} allows us to establish the following.

\begin{lemma}\label{lm:orb_int_est}
Let $n,F$ be as in Proposition \ref{prop:geometric_bound} and $\mathbf{m}$ be as in Lemma \ref{lm:weyl_elem}. We have
\begin{equation*}
    \mathcal{O}_{\boldsymbol{\psi}}(F,\widetilde{\mathbf{m}}c^{\ast}w_{\ast}\widetilde{\mathbf{m}}^{-1}) \ll_\epsilon \vert m\vert^{n-1+\epsilon}\frac{X^{\frac{1}{6}n(n-1)(n-2)+\epsilon}}{X^{n-1}} \cdot \vert c_1\cdots c_{n-1}\vert^{-1}\Vert F\Vert_{\infty}.
\end{equation*}
Furthermore, the left-hand side vanishes unless $s\ll \max(\frac{m}{rX^{n+1}},1)$.
\end{lemma}

\begin{proof}
We start with $c=(r,rs,\cdots,rs^{n-2})$. We have
\begin{equation}
    t_c := \widetilde{\mathbf{m}}c^{\ast}w_{\ast}\widetilde{\mathbf{m}}^{-1}w_{\ast}^{-1} = \left(\begin{matrix} \frac{m}{rs^{n-2}}  & & \\ & s  \I_{n-2} & \\ & & \frac{r}{m}\end{matrix}\right).\nonumber
\end{equation}
Write $a_X = \diag(X^{1-n} , \widetilde{a}_X,1)$ and recall that almost every $Y\in N_{\omega_{\ast}}$ can be written in the form
\begin{equation*}
    Y := \left(\begin{matrix} 1&  y\mathbf{x} \widetilde{a}_X & y \\ & \I_{n-2} & \widetilde{a}_X^{-1}\mathbf{z} \\ & & 1 \end{matrix}\right) \text{ for } \mathbf{x},\mathbf{z}^{\top}\in \R^{n-2} \text{ and }y\in \R^{\times}.
\end{equation*}
As we will be integrating over $N_{w_\ast}$ we will assume that $y\neq 0$. The measure on $N_{w_{\ast}}$ in these coordinates is given by $\d Y = \vert y\vert^{n-2} \d\mathbf{x}\d y\d\mathbf{z}$. Accordingly, we write
\begin{equation}
    r(\mathbf{x},y,\mathbf{z}) := a_Xw_{\ast}Ya_X^{-1} = \left(\begin{matrix} 0&0&-X^{1-n} \\ 0& \I_{n-2} & \mathbf{z} \\ X^{n-1} & y\mathbf{x} & y\end{matrix}\right). \nonumber
\end{equation}
By the definition of the orbital integral \eqref{eq:def-orb-int-gen} and after a change of variables in the $N$-integral we obtain
\begin{multline}
    \mathcal{O}_{\boldsymbol{\psi}}(F,\widetilde{\mathbf{m}}c^{\ast}w_{\ast}\widetilde{\mathbf{m}}^{-1}) \\ = X^{\frac{1}{6}n(n-1)(n-2)+\frac{1}{2}n(n-1)}\int_N \int_{N_{w_{\ast}}} \widetilde{F}(ut_cr(\mathbf{x},y,\mathbf{z}))\overline{\boldsymbol{\psi}_X(u)}e\left(-yx_1X^{2-n}-Xz_{n-1}\right)\d Y\d u, \nonumber
\end{multline}
where $\widetilde{F}(g) := F(a_X^{-1}ga_X).$ In particular, $\widetilde{F}$ is supported in $ZK(X,\tau)$. We now find a matrix $b\in B$ and $M\in B\cap \GL_{n-2}(\R)$ implicitly dependnent on $\mathbf{x},y,\mathbf{z}$ so that
\begin{equation}
    b\cdot r(\mathbf{x},y,\mathbf{z}) = \left(\begin{matrix} 1 & 0 & 0\\ -\frac{1}{yX^{1-n}} M\mathbf{z} & M(\I_{n-2}-\mathbf{z}\cdot \mathbf{x}) & 0\\ \frac{1}{X^{1-n}y} &\mathbf{x} & 1\end{matrix}\right) =: \overline{r}(\mathbf{x},y,\mathbf{z})\in N^\top.\nonumber
\end{equation}
Here the matrices $b$ and $M$ can be given explicitly as follows:
\begin{equation}
    b = \left(\begin{matrix} y(1-\mathbf{x}\cdot\mathbf{z}) & -X^{1-n}y\mathbf{x} & X^{1-n} \\ 0 & M & -\frac{1}{y} M\mathbf{z} \\ 0&0& \frac{1}{y} \end{matrix}\right),\quad M=\left(\begin{matrix} \frac{d_{n-2}}{d_{n-1}} & \frac{x_2z_1}{d_{n-1}} & \cdots & \frac{x_{n-3}z_1}{d_{n-1}} & \frac{x_{n-2}z_1}{d_{n-1}}\\  & \frac{d_{n-3}}{d_{n-2}} &  \ddots & \ddots & \frac{x_{n-2}z_2}{d_{n-2}} \\  &  & \ddots & \ddots & \vdots \\  &  &  & \frac{d_2}{d_3} & \frac{x_{n-2}z_{n-3}}{d_3} \\ & & & &\frac{1}{d_2}  \end{matrix}\right) \nonumber
\end{equation}
with $d_i=1-\sum_{j=n-i}^{n-2}x_jz_j$.
We write $b=t(\mathbf{x},y,\mathbf{z})^{-1}u(\mathbf{x},y,\mathbf{z})^{-1}\in A\times N$ and do a change of variable $ut_cu(\mathbf{x},y,\mathbf{z})t_c^{-1}\mapsto u$. Now we trivially estimate
\begin{multline*}
    \mathcal{O}_{\boldsymbol{\psi}}(F,\widetilde{\mathbf{m}}c^{\ast}w_{\ast}\widetilde{\mathbf{m}}^{-1})  \ll \Vert F\Vert_{\infty} X^{\frac{1}{6}n(n-1)(n-2)+\frac{1}{2}n(n-1)}\\
    \times\int_{N} \int_{N_{w_{\ast}}} \mathbf{1}_{K(X,\tau)}\left(ut_ct(\mathbf{x},y,\mathbf{z})\overline{r}(\mathbf{x},y,\mathbf{z})\right)\d Y\d u.
\end{multline*}
At this point we note that, for any $u\in N$, $t\in A$ and $\overline{u}\in N^\top$, the containment $ut\overline{u}\in K(X,\tau)$ implies $u,t,\overline{u}\in K(X,\tau')$ for a suitable $\tau' = \O(\tau)$. This is the archimedean analogue of the Iwahori decomposition and can for example be seen by using a suitable Lie-algebra decomposition. This allows us to estimate the $u$-integral by $\textrm{vol}(K(X,\tau)\cap N) \asymp X^{-\frac{1}{2}n(n-1)}$. We also  write $A(X,\tau):=K(X,\tau)\cap A$ and $N^\top(X,\tau) := K(X,\tau)\cap N^\top$. We get
\begin{multline}
    \mathcal{O}_{\boldsymbol{\psi}}(F,\widetilde{\mathbf{m}}c^{\ast}w_{\ast}\widetilde{\mathbf{m}}^{-1}) \ll \Vert F\Vert_{\infty} X^{\frac{1}{6}n(n-1)(n-2)}\\
    \times\int_{\R^{n-2}}\int_{\R}\int_{\R^{n-2}}  \mathbf{1}_{A(X,\tau')}(t_ct(\mathbf{x},y,\mathbf{z}))\mathbf{1}_{N^\top(X,\tau')}(\overline{r}(\mathbf{x},y,\mathbf{z})))\vert y\vert^{n-2}\d\mathbf{x}\d y\d\mathbf{z}. \nonumber
\end{multline}
Now we change $\mathbf{z}\mapsto\mathbf{z}'$ with $z'_{n-2}=1-z_{n-2}x_{n-2}$ and $z_i'=z_{i+1}'-x_iz_i$. This achieves that $d_i=z_{n-i}'$. We obtain
\begin{multline}
    \mathcal{O}_{\boldsymbol{\psi}}(F,\widetilde{\mathbf{m}}c^{\ast}w_{\ast}\widetilde{\mathbf{m}}^{-1}) \ll \Vert F\Vert_{\infty} X^{\frac{1}{6}n(n-1)(n-2)}\\
    \times\int_{\R^{n-2}}\int_{\R}\int_{\R^{n-2}}  \mathbf{1}_{A(X,\tau')}(t_ct(\mathbf{x},y,\mathbf{z}'))\mathbf{1}_{N^\top(X,\tau')}(\overline{r}(\mathbf{x},y,\mathbf{z}')))\vert y\vert^{n-2}\frac{\d\mathbf{x}\d y\d\mathbf{z'}}{\vert x_1\cdots x_{n-2}\vert}. \nonumber
\end{multline}
In particular, we compute that
\begin{equation}
    t(\mathbf{x},y,\mathbf{z}') = \diag\left(\frac{1}{yz_1'},\frac{z_1'}{z_2'},\ldots,z_{n-2}',y\right).\nonumber 
\end{equation}
This gives the conditions
\begin{equation}
    \left|\frac{yr}{m}-1\right|\le\frac{\tau'}{X},\quad \text{ and }\quad  \left|\frac{z_i'}{z'_{i+1}}s-1\right|\le \frac{\tau'}{X}\nonumber
\end{equation}
for $i=1,\ldots, n-2$ with the convention that $z_{n-1}'=1$.

We now check the support of the $y$-integral. Recall that the bottom left entry of $\overline{r}(\mathbf{x},y,\mathbf{z})$ is $\frac{X^{n-1}}{y}$, so we obtain
\begin{equation}
    \frac{X^{n-1}}{|y|} \le \frac{\tau'}{X}.\nonumber
\end{equation}
We record that combining the last two conditions on $y$ above we obtain the support constraint $r < \frac{m}{X^n}$ (which is not new to us; see Remark \ref{rem:w-ast-c-bound}).

We now concentrate on the $\mathbf{z}'$ and $\mathbf{x}$-integrals. Let $I$ be an interval so that $\frac{z_{i+1}'}{z_i's}\in I$. From the last estimates on the coordinates of $\mathbf{z}'$, we can surely assume that $I\subseteq [1-\tau''/X,1+\tau''/X]$ for some new $\tau''= \O(\tau)$. We do a recursive change of variables from $\mathbf{z}'\mapsto\mathbf{u}$ achieving $\frac{z_{i+1}'}{z_i'} = su_{i}$ for $i=1,\ldots, n-2$ and $\mathbf{x}\mapsto \frac{\mathbf{x}}{X}$. The Jacobian of this change of variable is
$\vert s\vert^{-\frac{1}{2}(n-1)(n-2)}\prod_{j=1}^{n-2} |u_j|^{\alpha_j}$ for some $\alpha_j\in\Z$. Noting that $u_j\asymp 1$ we obtain
\begin{multline}
    \mathcal{O}_{\boldsymbol{\psi}}(F,\widetilde{\mathbf{m}}c^{\ast}w_{\ast}\widetilde{\mathbf{m}}^{-1}) \ll \Vert F\Vert_{\infty} \vert s\vert^{-\frac{1}{2}(n-1)(n-2)}X^{\frac{1}{6}n(n-1)(n-2)}\\
    \times\int_{\R} \mathbf{1}_{[-\tau'/X,\tau'/X]}\left(\frac{yr}{m}-1\right)\vert y\vert^{n-2}\int_{I^{n-2}}\int_{\R^{n-2}} \mathbf{1}_{N^\top(X,\tau')}(\overline{n}(\mathbf{x},y,\mathbf{u}))\frac{\d\mathbf{x}\d\mathbf{u}}{\vert x_1\cdots x_{n-2}\vert}\d y. \nonumber
\end{multline}
with
\begin{equation}
    \overline{n}(\mathbf{x},y,\mathbf{u}) := 
    \left(\begin{matrix} 1 &  &  & &  &  \\
    \frac{X^n}{y x_1}(1-u_1s) & 1 &  &  &  &  \\
    \frac{X^n}{y x_2}(1-u_2s) & \frac{x_1}{x_2}(1-u_2s) & 1 &  &  &  \\
    \vdots & \ddots & \ddots & \ddots &  &  \\
    \vdots & \ddots & \ddots & 1 & & \\
    \frac{X^n}{y x_{n-2}}(1-u_{n-2}s) & \frac{x_1}{x_{n-2}}(1-u_{n-2}s) & \cdots & \frac{x_{n-3}}{x_{n-2}}(1-u_{n-2}s) & 1 & \\ \frac{X^{n-1}}{y } & \frac{x_1}{X} &\cdots  & \frac{x_{n-3}}{X} & \frac{x_{n-2}}{X} &  1 
    \end{matrix}\right).\nonumber
\end{equation}
The bottom row restricts $\mathbf{x} \in [-\tau',\tau']^{n-2}$. Now we make sure that the $\mathbf{x}$-integral converges. 

If $s\neq 1$, then the first column gives the bound
\begin{equation*}
    \frac{1}{x_i} \ll \frac{y}{s X^{n+1}}\asymp\frac{m}{rsX^{n+1}}.
\end{equation*}
Combining this with the support of the $x_i$-integrals obtained from the bottom row we conclude that
\begin{equation}
    \frac{rs}{m}X^{n+1} \ll x_i \ll 1. \nonumber
\end{equation}
From this we extract the required upper bound on $s$.

At this end, we note that the contribution of the $y$-integral is
\begin{equation}
    \int_{\R} \mathbf{1}_{[-\tau'/X,\tau'/X]}\left(\frac{yr}{m}-1\right)\vert y\vert^{n-2}\d y \ll_\tau \frac{1}{X}\left\vert \frac{m}{r}\right\vert^{n-1}. \label{eq:y_est}
\end{equation}
Estimating the remaining integrals trivially and noting that $\vol(I)\asymp X^{-1}$ we obtain the desired bound in this case.

Finally, we consider $s=1$. We change variables $\mathbf{u}\mapsto\mathbf{u'}$ with $u_i =  1+x_iu_i'$. Note that $u_i'\in [-\frac{\tau''}{|x_i|X},\frac{\tau''}{x_iX}]$. Thus after defining
\begin{equation}
    \mathcal{I}(\mathbf{x}) := \prod_{i=1}^{n-2} \left[-\frac{\tau''}{|x_i|X},\frac{\tau''}{|x_i|X}\right] \subseteq \R^{n-2},\nonumber
\end{equation}
we arrive at
\begin{multline}
    \mathcal{O}_{\boldsymbol{\psi}}(F,\widetilde{\mathbf{m}}c^{\ast}w_{\ast}\widetilde{\mathbf{m}}^{-1}) \ll \Vert F\Vert_{\infty}  X^{\frac{1}{6}n(n-1)(n-2)}\int_{\R} \mathbf{1}_{[-\tau'/X,\tau'/X]}\left(\frac{yr}{m}-1\right)\vert y\vert^{n-2}\\
    \int_{[-\tau',\tau']^{n-2}} \int_{\mathcal{I}(\mathbf{x})}\mathbf{1}_{N^\top(X,\tau')}(\overline{n}(\mathbf{x},y,\mathbf{u}'))\d\mathbf{u}'\d\mathbf{x}\d y. \nonumber
\end{multline}
with
\begin{equation}
    \overline{n}(\mathbf{x},y,\mathbf{u}') = 
    \left(\begin{matrix} 1 &  &  &  &  &  \\
    \frac{X^n}{y}u_1' & 1 &  &  &  &   \\
    \frac{X^n}{y}u_2' & x_1u_2' & 1 &  &  &   \\
    \vdots & \ddots & \ddots & &  &   \\
    \vdots & \ddots & \ddots & 1 & & \\
    \frac{X^n}{y}u_{n-2}' & x_1u_{n-2}' & \cdots & x_{n-3}u_{n-2}' & 1 & \\ \frac{X^{n-1}}{y} & \frac{x_1}{X} &\cdots  & \frac{x_{n-3}}{X} & \frac{x_{n-2}}{X} &  1 
    \end{matrix}\right).\nonumber
\end{equation}
First, recall that $y= \frac{m}{r}\left(1+\O(X^{-1})\right)$. Thus the first column of the matrix $\overline{n}(\mathbf{x},y,\mathbf{u}')$ gives the condition
\begin{equation}
    u_i' \ll \frac{m}{r X^{n+1}} \text{ for } i=1,\ldots, n-2.\nonumber
\end{equation}
On the other hand, rows $3$ to $n-1$ together with $\mathbf{u}'\in \mathcal{I}(\mathbf{x})$ give the conditions
\begin{equation}
    u_i' \ll \frac{1}{X \max(|x_1|,\ldots,|x_{i}|)} \text{ for }i=1,\ldots, n-2.\nonumber
\end{equation}
At this point we can execute the $\mathbf{u}'$-integral and the $y$-integral. We obtain
\begin{equation}
    \mathcal{O}_{\boldsymbol{\psi}}(F,\widetilde{\mathbf{m}}c^{\ast}w_{\ast}\widetilde{\mathbf{m}}^{-1}) \ll \Vert F\Vert_{\infty} \left\vert \frac{m}{r}\right\vert^{n-1}
    X^{\frac{1}{6}n(n-1)(n-2)-1}
    \int_{[-\tau',\tau']^{n-2}} \prod_{i=1}^{n-2}\vol(B_i(\mathbf{x}))  \d\mathbf{x}, \nonumber
\end{equation}
where (after possibly modifying the implied constants)
\begin{equation}
    B_i(\mathbf{x}) := \left\{ u_i'\in \R \,:\, |u_i'|\,\le \min\left(\frac{m}{rX^{n+1}},\frac{1}{|x_1|X},\ldots,\frac{1}{|x_{i}|X}\right)\right\}. \nonumber
\end{equation}
The remaining task is to estimate the $\mathbf{x}$-integral above. We majorize this integral by $\sum_{j=0}^{n-2}\mathcal{I}_j$, where $\mathcal{I}_j$ is the integral over the range of $\mathbf{x}$ with
\begin{equation*}
    B_1(\mathbf{x}) = \ldots = B_j(\mathbf{x}) = \left[-\frac{m}{rX^{n+1}},\frac{m}{rX^{n+1}}\right]\supset \left[-\frac{1}{|x_j|X},\frac{1}{|x_j|X}\right] = B_{j+1}(\mathbf{x}).
\end{equation*}
In other words, the integral $\mathcal{I}_j$ is over all $\mathbf{x}$ such that
\begin{equation}\label{eq:condition}
    |x_i|\,\le \frac{rX^n}{m}\text{ for }i=1,\ldots, j;\quad\text{and}\quad |x_{j+1}|\,\ge \frac{rX^n}{m}.
\end{equation}
Note that \eqref{eq:condition} allows us to estimate the integral $\mathcal{I}_j$ as follows:
\begin{equation*}
    \mathcal{I}_j \ll  X^{-(n-2)} \intop_{\frac{rX^n}{m}\ll x_{j+1}\ll 1}\intop_{x_{j+2}\ll 1}\dots \intop_{x_{n-2}\ll 1} \\\prod_{i=j+1}^{n-2}\max\left(|x_{j+1}|,\ldots,|x_{i}|\right)^{-1} \d x_{j+1} \dots \d x_{n-2} .\nonumber
\end{equation*}
Starting from the $x_{n-2}$-integral and working backwards we now estimate that that the remaining integral is $\O_\epsilon((rmX)^{\epsilon})$. This can be checked by considering the two cases $|x_{l}|<\max\left(|x_{j+1}|,\ldots, |x_{l-1}|\right)$ and $|x_{l}|>\max\left(|x_{j+1}|,\ldots, |x_{l-1}|\right)$ separately.

The case where $c=(rs^{n-2}, \ldots, rs,r)$ can be reduced to the case above by using a suitable involution as in the proof of Lemma~\ref{lm:weyl_elem}.
\end{proof}

\begin{proof}[Proof of Proposition \ref{prop:geometric_bound}]
Note that our assumptions are tailored for an application of Corollary~\ref{cor:support_geo}, so that using \eqref{eq:value-delta} and recalling Remark \ref{rem:w-ast-c-bound} we get
\begin{multline*}
    \sum_{w\in W}\sum_{\substack{0\neq c\in \Z^{n-1} \\ \text{admissible}}} \frac{S_w(\mathbf{m},\mathbf{m};c)}{\delta_w(\widetilde{\mathbf{m}})} \mathcal{O}_{\boldsymbol{\psi}}(F,\widetilde{\mathbf{m}}c^{\ast}w\widetilde{\mathbf{m}}^{-1}) \\= W_F(1) + \sum_{\substack{c\in \Z_{\neq 0}^{n-1} \\ 
    c_1<\frac{m}{X^n} 
    \\ \mathrm{admissible}}}\frac{S_{w_{\ast}}(\mathbf{m},\mathbf{m};c)}{m^{n-1}}\mathcal{O}_{\boldsymbol{\psi}}(F;\widetilde{\mathbf{m}}c^{\ast}w\widetilde{\mathbf{m}}^{-1}).
\end{multline*}
In view of Lemma~\ref{lm:trivial_identity} it is sufficient to estimate
\begin{equation}
    S:=\sum_{\substack{c\in \Z_{\neq 0}^{n-1} \\ c_1<\frac{m}{X^n} \\ \text{admissible}}}\frac{S_{w_{\ast}}(\mathbf{m},\mathbf{m};c)}{m^{n-1}}\mathcal{O}_{\boldsymbol{\psi}}(F;\widetilde{\mathbf{m}}c^{\ast}w_{\ast}\widetilde{\mathbf{m}}^{-1}).
\end{equation}
We estimate the orbital integral using Lemma~\ref{lm:orb_int_est}. Writing $c(r,s)$ for the admissible moduli, as in \eqref{eq:admis_c}, we obtain
\begin{equation}
    S \ll \Vert F\Vert_{\infty} X^{\frac{1}{6}n(n-1)(n-2)+\epsilon}X^{1-n}\sum_{r\ll \frac{m}{X^n}} \sum_{s\ll \max(\frac{m}{rX^{n+1}},1)} \frac{\vert S_{w_{\ast}}(\mathbf{m},\mathbf{m};c(r,s))\vert}{\vert r^{n-1}s^{\frac{1}{2}(n-1)(n-2)}\vert}.\nonumber
\end{equation}
Estimating the Kloosterman sum trivially, as in \eqref{eq:triv-bound-klooster}, the result follows directly.
\end{proof}

\section{Proof of the Main Theorems}\label{sec:proof}

We have now gathered everything to derive the main result. We first combine the spectral and the geometric estimate to get the following key lemma, which essentially entails \eqref{eq:Hecke_Version}. We fix $n\ge 4$ and a prime $p$. Also, we assume that $X>1$ is tending to infinity and $\tau>0$ sufficiently small but fixed, throughout the section.

\begin{lemma}\label{lm:Hecke_bound}
Let $\mathbf{m}:=(m,1\dots,1)$ such that $m\ll X^{n+2}$. Then we have
\begin{equation}
    \int_{\substack{\widehat{\mathbb{X}}_{\gen}\\ \Vert \mu_{\infty}(\pi)\Vert \ll X}} \vert \lambda_{\pi}(\mathbf{m})\vert^2\d\mu_{\aut}(\pi) \ll_\epsilon X^{\frac{1}{2}(n+2)(n-1)+\epsilon}.\nonumber
\end{equation}
\end{lemma}

\begin{proof}
We define 
\begin{equation}
    F(g) = \int_{Z}F_X^{\sharp}(zg)\d z,\nonumber 
\end{equation}
where $F_X^{\sharp}$ is as in \eqref{eq:def_FX}. We observe that this defines a function $F\in \mathcal{C}_c^{\infty}(Z\backslash G)$ with support in $Z K(X,\tau)^{\sharp}$ and
\begin{equation}
    \Vert F\Vert_{\infty}\, \ll X^{n^2-1}\nonumber
\end{equation}
which follow from \eqref{eq:support-F} and \eqref{eq:supnorm-average}.
Applying the Kuznetsov formula, namely Proposition~\ref{prop:Kuznetsov}, with the test function $F$ and estimating the geometric side using Proposition~\ref{prop:geometric_bound} yields
\begin{equation}
    \int_{\widehat{\mathbb{X}}_{\textrm{gen}}} \vert \lambda_{\pi}(\mathbf{m})\vert^2J_{\pi,\boldsymbol{\psi}}(\overline{F})\frac{\d\mu_{\mathrm{aut}}(\pi)}{\ell(\pi)} \ll  X^{n^2-1+\frac{1}{6}n(n-1)(n-2)+\epsilon}.\nonumber
\end{equation}
By positivity we can drop all representations with $\Vert \mu_{\infty}(\pi) \Vert \gg X$ from the left hand side and apply Proposition~\ref{prop:spectral_lower_bound}. This yields 
\begin{equation}
    X^{\frac{1}{6}n(n-1)(n-2)+\frac{1}{2}n(n-1)}\int_{\substack{\widehat{\mathbb{X}}_{\textrm{gen}}\\ \Vert \mu_{\infty}(\pi)\Vert \ll X}} \vert \lambda_{\pi}(\mathbf{m})\vert^2\frac{\d\mu_{\mathrm{aut}}(\pi)}{\ell(\pi)} \ll  X^{n^2-1+\frac{1}{6}n(n-1)(n-2)+\epsilon}.\nonumber
\end{equation}
Using $\ell(\pi)\asymp L(1,\pi,\textrm{Ad})\ll \Vert \mu_{\infty}(\pi)\Vert^{\epsilon}$ from \cite{li2010upper} to remove the harmonic weights and dividing both sides by $X^{\frac{1}{6}n(n-1)(n-2)+\frac{1}{2}n(n-1)}$ yields the desired result.
\end{proof}

We use this to produce the following result.

\begin{lemma}\label{lm:standard_step}
We have
\begin{equation*}
    \sum_{\pi\in \Omega_{\cusp}(X)}p^{2r\sigma_p(\pi)} \ll_p X^{\frac{1}{2}(n+2)(n-1)+\epsilon},
\end{equation*}
for $p^r\ll X^{n+2}$.
\end{lemma}

\begin{proof}
From \cite[(4.3)]{jana2021application} we recall that
\begin{equation}
    \sum_{j=0}^{n-1} \vert \lambda_{\pi}((p^{r-j},1,\ldots,1))\vert^2 \,\ge 2p^{1-n}p^{2r\sigma_p(\pi)}, \nonumber
\end{equation}
for $r>n$ and $\pi$ cuspidal. The result follows directly from Lemma~\ref{lm:Hecke_bound} after dropping all non-cuspidal terms.
\end{proof}

We are finally ready to complete the proof of our main results.

\begin{proof}[Proof of Theorem~\ref{th:main_density}]
We start by choosing $r\in \N$ such that $p^r\asymp_p X^{n+2}$. From Lemma~\ref{lm:standard_step} we obtain
\begin{equation*}
    \sum_{\pi\in \Omega_{\textrm{cusp}}(X)}X^{(n+2)2\sigma_p(\pi)} \ll_{p,\epsilon} X^{\frac{1}{2}(n+2)(n-1)+\epsilon}.
\end{equation*}
At this point we apply the Rankin trick to get
\begin{equation}
    \# \{ \pi\in \Omega_{\textrm{cusp}}(X)\colon \sigma_p(\pi)\geq \sigma\} \leq X^{-2\sigma(n+2)} \sum_{\pi\in \Omega_{\textrm{cusp}}(X)}X^{(n+2)2\sigma_p(\Pi)} \ll_p X^{\frac{1}{2}(n+2)(n-1)-2\sigma(n+2)+\epsilon}. \nonumber
\end{equation}
The desired result follows immediately after re-arranging the exponent.
\end{proof}

\begin{proof}[Proof of Theorem~\ref{th:optimal-exp}]
The family $\mathcal{F}_X$ defined in \cite[(4.5)]{jana2024optimal} is in bijection with  the spherical sub-family $$\{ \pi\in \Omega_{\textrm{cusp}}(X)\colon \pi \text{ spherical}\}.$$ Furthermore, under this bijection our $\sigma_p(\pi)$ corresponds to $\theta_{\varphi,p}$ defined in \emph{loc.cit.}. In particular, Theorem~\ref{th:main_density} directly implies \cite[Conjecture~2]{jana2024optimal}. This makes \cite[Theorem~3]{jana2024optimal} unconditional and establishes the desired result.
\end{proof}

\section*{Acknowledgement}
We thank Paul Nelson for his feedback on an earlier draft of this paper. Furthermore, we thank Valentin Blomer for his encouragement and many interesting discussions.

\bibliographystyle{abbrv}
\bibliography{database.bib}

\begin{thebibliography}{10}

\bibitem{assing2024principal}
E.~Assing and V.~Blomer.
\newblock The density conjecture for principal congruence subgroups.
\newblock {\em Duke Math. J.}, 173(7):1359--1426, 2024.

\bibitem{assing2024density}
E.~Assing, V.~Blomer, and P.~D. Nelson.
\newblock Local analysis of the kuznetsov formula and the density conjecture.
\newblock {\em arXiv}, 2404.05561, 2024.

\bibitem{blomer2013applications}
V.~Blomer.
\newblock Applications of the {K}uznetsov formula on {${\rm GL}(3)$}.
\newblock {\em Invent. Math.}, 194(3):673--729, 2013.

\bibitem{blomer2019density}
V.~Blomer.
\newblock Density theorems for {${\rm GL}(n)$}.
\newblock {\em Invent. Math.}, 232(2):683--711, 2023.

\bibitem{blomer2011ramanujan}
V.~Blomer and F.~Brumley.
\newblock On the {R}amanujan conjecture over number fields.
\newblock {\em Ann. of Math. (2)}, 174(1):581--605, 2011.

\bibitem{blomer2014sato}
V.~Blomer, J.~Buttcane, and N.~Raulf.
\newblock A {S}ato-{T}ate law for {$\rm GL (3)$}.
\newblock {\em Comment. Math. Helv.}, 89(4):895--919, 2014.

\bibitem{fraczyk2024optimal}
M.~Fraczyk, A.~Gorodnik, and A.~Nevo.
\newblock Automorphic density estimates and optimal diophantine exponents.
\newblock {\em arXiv}, 2402.15875, 2024.

\bibitem{fraczyk2024density}
M.~Fraczyk, G.~Harcos, P.~Maga, and D.~Mili\'cevi\'c.
\newblock The density hypothesis for horizontal families of lattices.
\newblock {\em Amer. J. Math.}, 146(1):107--160, 2024.

\bibitem{getz2024intro}
J.~R. Getz and H.~Hahn.
\newblock {\em An introduction to automorphic representations---with a view
  toward trace formulae}, volume 300 of {\em Graduate Texts in Mathematics}.
\newblock Springer, Cham, [2024] \copyright 2024.

\bibitem{ghosh2015diophantine}
A.~Ghosh, A.~Gorodnik, and A.~Nevo.
\newblock Diophantine approximation exponents on homogeneous varieties.
\newblock In {\em Recent trends in ergodic theory and dynamical systems},
  volume 631 of {\em Contemp. Math.}, pages 181--200. Amer. Math. Soc.,
  Providence, RI, 2015.

\bibitem{ghosh2018best}
A.~Ghosh, A.~Gorodnik, and A.~Nevo.
\newblock Best possible rates of distribution of dense lattice orbits in
  homogeneous spaces.
\newblock {\em J. Reine Angew. Math.}, 745:155--188, 2018.

\bibitem{humphries2018density}
P.~Humphries.
\newblock Density theorems for exceptional eigenvalues for congruence
  subgroups.
\newblock {\em Algebra Number Theory}, 12(7):1581--1610, 2018.

\bibitem{iwaniec1990small}
H.~Iwaniec.
\newblock Small eigenvalues of {L}aplacian for {$\Gamma_0(N)$}.
\newblock {\em Acta Arith.}, 56(1):65--82, 1990.

\bibitem{jana2021application}
S.~Jana.
\newblock Applications of analytic newvectors for {$\text{GL}(n)$}.
\newblock {\em Math. Ann.}, 380(3-4):915--952, 2021.

\bibitem{jana2024local-l2}
S.~Jana and A.~Kamber.
\newblock On the local {$L^2$}-{B}ound of the {E}isenstein series.
\newblock {\em Forum Math. Sigma}, 12:Paper No. e76, 2024.

\bibitem{jana2024optimal}
S.~Jana and A.~Kamber.
\newblock Optimal {D}iophantine exponents for {SL}({$n$}).
\newblock {\em Adv. Math.}, 443:Paper No. 109613, 2024.

\bibitem{jana2020analytic}
S.~Jana and P.~D. Nelson.
\newblock Analytic newvectors for $\mathrm{GL}_n(\mathbb{R})$.
\newblock {\em arXiv}, 1911.01880, 2020.

\bibitem{lapid2015whittaker}
E.~Lapid and Z.~Mao.
\newblock A conjecture on {W}hittaker-{F}ourier coefficients of cusp forms.
\newblock {\em J. Number Theory}, 146:448--505, 2015.

\bibitem{li2010upper}
X.~Li.
\newblock Upper bounds on {$L$}-functions at the edge of the critical strip.
\newblock {\em Int. Math. Res. Not. IMRN}, 0(4):727--755, 2010.

\bibitem{linn2024kloosterman}
J.~Linn.
\newblock Bounds for kloosterman sums for $\mathrm{GL}_n$.
\newblock {\em arXiv}, 2412.04976, 2024.

\bibitem{matz2021sato-tate}
J.~Matz and N.~Templier.
\newblock Sato-{T}ate equidistribution for families of {H}ecke-{M}aass forms on
  {${\mathrm {SL}}(n, \mathbb{R} )\slash {\rm SO}(n)$}.
\newblock {\em Algebra Number Theory}, 15(6):1343--1428, 2021.

\bibitem{muller2007weyl}
W.~M\"{u}ller.
\newblock Weyl's law for the cuspidal spectrum of {${\rm SL}_n$}.
\newblock {\em Ann. of Math. (2)}, 165(1):275--333, 2007.

\bibitem{nelson2023standard}
P.~D. Nelson.
\newblock Bounds for standard $l$-functions.
\newblock {\em arXiv}, 2109.15230, 2023.

\bibitem{nelson2023spectral}
P.~D. Nelson.
\newblock Spectral aspect subconvex bounds for {$\mathrm{U}_{n+1} \times
  \mathrm{U}_n$}.
\newblock {\em Invent. Math.}, 232(3):1273--1438, 2023.

\bibitem{nelson2021orbit}
P.~D. Nelson and A.~Venkatesh.
\newblock The orbit method and analysis of automorphic forms.
\newblock {\em Acta Math.}, 226(1):1--209, 2021.

\bibitem{sarnak1991bounds}
P.~Sarnak and X.~X. Xue.
\newblock Bounds for multiplicities of automorphic representations.
\newblock {\em Duke Math. J.}, 64(1):207--227, 1991.

\end{thebibliography}

\end{document}